\documentclass[reqno, a4paper]{amsart}

\usepackage[english]{babel}
\usepackage[a4paper,top=3cm,bottom=3cm,left=3cm,right=3cm, marginparwidth=1.7cm]{geometry}
\usepackage{amsmath}
\usepackage{amsfonts}
\usepackage{amssymb}
\usepackage{amsthm}
\usepackage{tikz}
\usepackage{mathrsfs}
\usepackage{mathtools}
\usepackage{enumitem}
\usepackage{url}
\usepackage{bbm}
\usepackage{tikz-cd}
\usepackage{stmaryrd}

\usetikzlibrary{arrows}

\newcommand{\N}{\mathbb{N}}
\newcommand{\Z}{\mathbb{Z}}
\newcommand{\Q}{\mathbb{Q}}
\newcommand{\R}{\mathbb{R}}

\renewcommand{\leq}{\leqslant}
\renewcommand{\geq}{\geqslant}
\newcommand{\seq}{\subseteq}
\newcommand{\df}{\coloneqq}
\newcommand{\ind}{\mathbbm{1}}

\newcommand{\RSu}{\sigma\mathbb{RS}_u}
\newcommand{\Gt}{\sigma\ell\mathbb{G}_t}
\newcommand{\RSt}{\sigma\mathbb{RS}_t}
\newcommand{\lpmu}{\mathcal{L}^p(\mu)}

\newcommand{\truncsup}{\bigcurlyvee}
\newcommand{\trunc}{\overline{\;\cdot\;}}
\newcommand{\RXI}{\frac{\R^X}{\mathcal{I}}}

\let\originalleft\left
\let\originalright\right
\renewcommand{\left}{\mathopen{}\mathclose\bgroup\originalleft}
\renewcommand{\right}{\aftergroup\egroup\originalright}
\let\originallvert\lvert
\renewcommand{\lvert}{\left\originallvert}
\let\originalrvert\rvert
\renewcommand{\rvert}{\right\originalrvert}

\newcommand{\Cyl}{\mathrm{Cyl}}
\newcommand{\Leb}{\mathrm{Leb}}

\newtheorem{theorem}{Theorem}[section]
\newtheorem{lemma}[theorem]{Lemma}
\newtheorem{corollary}[theorem]{Corollary}
\newtheorem{proposition}[theorem]{Proposition}

\theoremstyle{definition}
\newtheorem{definition}[theorem]{Definition}
\newtheorem{question}[theorem]{Question}

\theoremstyle{remark}
\newtheorem{remark}[theorem]{Remark}
\newtheorem{example}[theorem]{Example}

\title{Operations that preserve integrability, and truncated Riesz spaces}

\author[Marco Abbadini]{Marco Abbadini}
\email{marco.abbadini@unimi.it}
\address{Dipartimento di Matematica {\sl Federigo Enriques}, Universit\`a degli Studi di Milano, via Cesare Saldini 50, 20133 Milano, Italy.}


\begin{document}

\maketitle
\begin{abstract}
	For any real number $p\in [1,+\infty)$, we characterise the operations $\R^I\to \R$ that preserve $p$-integrability, i.e., the operations under which, for every measure $\mu$, the set $\mathcal{L}^p(\mu)$ is closed.
	We investigate the infinitary variety of algebras whose operations are exactly such functions.
	It turns out that this variety coincides with the category of Dedekind $\sigma$-complete truncated Riesz spaces, where truncation is meant in the sense of R.\ N.\ Ball.
	We also prove that $\R$ generates this variety.
	From this, we exhibit a concrete model of the free Dedekind $\sigma$-complete truncated Riesz spaces.
	
	Analogous results are obtained for operations that preserve $p$-integrability over finite measure spaces: the corresponding variety is shown to coincide with the much studied category of Dedekind $\sigma$-complete Riesz spaces with weak unit, $\R$ is proved to generate this variety, and a concrete model of the free Dedekind $\sigma$-complete Riesz spaces with weak unit is exhibited.
\end{abstract}
%

%%%%%%%%%%%%%%%%%%%%%%%%%%%%%%%%%%%%%%%%%%%%%%%%%%%%%%%%%%%%%%%%%%%%%%%%%%%%%%%%%%%%
%SECTION
\section{Introduction}

%===================================================================================
%SUBSECTION
\subsection{Operations that preserve integrability}

In this work we investigate the operations which are somehow implicit in the theory of integration by addressing the following question: which operations preserve integrability, in the sense that they return integrable functions when applied to integrable functions?
%%
%\begin{enumerate} [label=(Q\arabic*), ref = Q\arabic*]
%%
%	\item \label{i:quest-preserve} What are the operations which preserve integrability, in the sense that they return integrable functions when applied to integrable functions?
%	%
%%	\item \label{i:quest-generating} Assuming an answer to \eqref{i:quest-preserve}, what is a simple generating set for the operations in \eqref{i:quest-preserve}?
%%	%
%%	\item \label{i:quest-axioms} Assuming an answer to \eqref{i:quest-generating}, what are the basic axioms satisfied by the operations in \eqref{i:quest-generating}?
%	%
%\end{enumerate}
%
%In Part 1 of this paper we provide an answer to \eqref{i:quest-preserve} and \eqref{i:quest-generating}, whereas in Part 2, which is more algebraic, we address \eqref{i:quest-axioms}.

Let us clarify the question by recalling some definitions.

For $(\Omega,\mathcal{F},\mu)$ a measure space (with the range of $\mu$ in $[0,+\infty]$) and $p\in [1,+\infty)$, we adopt the notation
$\lpmu\df \{f\colon \Omega\to \R\mid f\text{ is }\mathcal{F}\text{-measurable and }$ $\int_{\Omega}\lvert f\rvert ^p\,d\mu<\infty \}.$
It is well known that, for $f,g\in \lpmu$, we have $f+g\in \lpmu$, that is, $\lpmu$ is closed under the pointwise addition induced by addition of real numbers $+\colon \R^2\to \R$.
More generally, consider a set $I$ and a function $\tau\colon \R^I\to \R$, which we shall call an \emph{operation of arity $\lvert I\rvert$}.
We say \emph{$\lpmu$ is closed under $\tau$} if $\tau$ returns functions in $\lpmu$ when applied to functions in $\lpmu$, that is, for every $(f_i)_{i\in I}\seq \lpmu$, the function $\tau((f_i)_{i\in I})\colon \Omega\to \R$ given by $x\in \Omega\mapsto \tau((f_i (x))_{i\in I})$ belongs to $\lpmu$.
If $\lpmu$ is closed under $\tau$, we also say that $\tau$ \emph{preserves $p$-integrability over $(\Omega,\mathcal{F},\mu)$}.
Finally, we say that $\tau$ \emph{preserves $p$-integrability} if $\tau$ preserves $p$-integrability over every measure space.

In Part \ref{part:op} of this paper we characterise those operations that preserve integrability.
Indeed, the first question we address is the following.
\begin{question}\label{q:main}
	Under which operations $\R^I\to \R$ are $\mathcal{L}^p$ spaces closed?
	Equivalently, which operations preserve $p$-integrability?
\end{question}
\noindent Examples of such operations are the constant $0$, the addition $+$, the binary supremum $\lor$ and infimum $\land$, and, for  $\lambda \in \R$,  the  scalar multiplication $\lambda ({\;\cdot\;})$ by $\lambda$.
 A further example is the operation of countably infinite arity $\truncsup$ defined as 
\[
	\truncsup(y,x_0,x_1,\dots)\df\sup_{n\in \omega }\{x_n\land y \}.
\]
Yet another example is the unary operation 
\begin{align*}
	\overline{\;\cdot\;} \colon	\R	& \longrightarrow	\R \\
								x	& \longmapsto		\overline{x} \df x \land 1,
\end{align*}
called \emph{truncation}.
Here, although the constant function $1\in\lpmu$ if, and only if, $\mu$ is finite, it is always the case that $f\in \lpmu$ implies $\overline{f}\in\lpmu$.

It turns out that, for any given $p$, the operations that preserve $p$-integrability are essentially just $0$, $+$, $\lor$, $\lambda ({\;\cdot\;})$ (for each $\lambda \in \R$),  $\truncsup$ and  $\trunc$, in the sense that every operation that preserves $p$-integrability may be obtained from these by composition.
This we prove in Theorem~\ref{t:generation}.

We also have an explicit characterisation of the operations that preserve $p$-integrability.
Denoting with $\R^+$ the set $\{\lambda \in \R \mid \lambda \geq 0\}$, for $n\in\omega$ and $\tau\colon \R^n\to \R$, we will prove that $\tau$ preserves $p$-integrability precisely when $\tau$ is Borel measurable and there exist $\lambda_0,\dots,\lambda_{n-1}\in \R^+$ such that, for every $x\in \R^n$, we have
\[
	\lvert \tau(x) \rvert \leq \sum_{i=0}^{n-1}\lambda_i \lvert x_i \rvert.
\]
Theorem~\ref{t:MAIN} tackles the general case of arbitrary arity, settling Question \ref{q:main}.

In Part I we also address a variation of Question \ref{q:main} where we restrict attention to finite measures.
Recall that a measure $\mu$ on a measurable space $(\Omega, \mathcal{F})$ is \emph{finite} if $\mu(\Omega)<\infty$.
The question becomes
\begin{question}\label{q:main finite}
	Under which operations $\R^I\to \R$ are $\mathcal{L}^p$ spaces of finite measure closed? Equivalently, which operations preserve $p$-integrability over finite measure spaces?
\end{question}
\noindent As mentioned, the function constantly equal to $1$ belongs to $\lpmu$ for every finite measure $\mu$.	We prove in Theorem~\ref{t:generation finite} that, for any given $p \in [1, +\infty)$, the operations that preserve $p$-integrability over finite measure spaces are essentially just $0$, $+$, $\lor$, $\lambda ({\;\cdot\;})$ (for each $\lambda \in \R$),  $\truncsup$ and  $1$, in the same sense as in the above.

Theorem~\ref{t:MAIN finite} provides an explicit characterisation of the operations that preserve $p$-integrability over finite measure spaces.
In particular, for $n\in\omega$ and $\tau\colon \R^n\to \R$, $\tau$ preserves $p$-integrability over finite measure spaces precisely when $\tau$ is Borel measurable and there exist $\lambda_0,\dots,\lambda_{n-1}, k\in \R^+$ such that, for every $x\in \R^n$, we have
\[
	\lvert \tau(x) \rvert \leq k + \sum_{i = 0}^{n-1}\lambda_i \lvert x_i \rvert.
\]
%

%===================================================================================
%SUBSECTION
\subsection{Truncated Riesz spaces and weak units}

In Part~\ref{part:variety} of this paper we investigate the equational laws satisfied by the operations that preserve $p$-integrability.
(As it is shown by Theorems~\ref{t:MAIN} and~\ref{t:MAIN finite}, the fact that an operation preserves $p$-integrability -- over arbitrary and finite measure spaces, respectively -- does not depend on the choice of $p$.
Hence, we say that the operation \emph{preserves integrability}.)
We therefore work in the setting of varieties of algebras \cite{Burris_Sanka}.
In this paper, under the term \emph{variety} we include also infinitary varieties, i.e.\ varieties admitting primitive operations of infinite arity.
For background please see \cite{Slominski}.

We assume familiarity with the basic theory of Riesz spaces, also known as vector lattices.
All needed background can be found, for example, in the standard reference \cite{LuxZaan}.
As usual, for a Riesz space $G$, we set $G^+\df\{x\in G\mid x\geq 0\}$.

A \emph{truncated Riesz space} is a Riesz space $G$ endowed with a function $\overline{\;\cdot\;} \colon G^+ \to G^+$, called \emph{truncation}, which satisfies the following properties for all $f,g\in G^+$.
\begin{enumerate}[label = (B\arabic*)]
	\item $f\land \overline{g}\leq \overline{f}\leq f$.
	\item If $\overline{f}=0$, then $f=0$.
	\item If $nf =\overline{nf}$ for every $n\in \omega$, then $f=0$.
\end{enumerate}
The notion of truncation is due to R.\ N.\ Ball \cite{Ball1}, who introduced it in the context of lattice-ordered groups.
Please see Section~\ref{s:defn truncated} for further details.
	
Let us say that a partially ordered set $B$ is \emph{Dedekind $\sigma$-complete} if every nonempty countable subset $A\seq B$ that admits an upper bound admits a supremum.
Theorem~\ref{t:variety is truncated} proves that the category of Dedekind $\sigma$-complete truncated Riesz spaces is a variety generated by $\R$.
This variety can be presented as having operations of finite arity only, together with the single operation $\truncsup$ of countably infinite arity.
Moreover, we prove that the variety is finitely axiomatisable by equations over the theory of Riesz spaces.
One consequence (Corollary~\ref{c:main theorem 2}) is that the free Dedekind $\sigma$-complete truncated Riesz space over a set $I$ (exists, and) is 
\[
	\mathrm{F_t}(I)\df\left\{f\colon \R^I\to \R\mid f\text{ preserves integrability} \right\}.
\]

We prove results analogous to the foregoing for operations that preserve integrability over finite measure spaces.
An element $1$ of a Riesz space $G$ is a \emph{weak \textup{(}order\textup{)} unit} if $1\geq0$ and, for all $f\in G$, $f\land 1 =0$ implies $ f=0$.
Theorem~\ref{t:variety is truncated weak} shows that the category of Dedekind $\sigma$-complete Riesz spaces with weak unit is a variety generated by $\R$, again with primitive operations of countable arity.
It, too, is finitely axiomatisable by equations over the theory of Riesz spaces.
By Corollary~\ref{c:main theorem 1}, the free Dedekind $\sigma$-complete Riesz space with weak unit over a set $I$ (exists, and) is 
\[
	\mathrm{F_u}(I)\df\left\{f\colon \R^I\to \R\mid f\text{ preserves integrability over finite measure spaces} \right\}.
\]

The varietal presentation of Dedekind $\sigma$-complete Riesz spaces with weak unit was already obtained in \cite{Abba}.
Here we add the representation theorem for free algebras, and we establish the relationship between Dedekind $\sigma$-complete Riesz spaces with weak unit and operations that preserve integrability.
The proofs in the present paper are independent of \cite{Abba}.
On the other hand, the results in this paper do depend on a version of the Loomis-Sikorski Theorem for Riesz spaces, namely Theorem~\ref{t:Nepalese} below.
A proof can be found  in \cite{Nepalese},  and can also be recovered  from the combination of \cite{LoomisRevisited} and \cite{SmallRieszSpaces}.
The theorem and its variants have a long history: for a fuller bibliographic account please see \cite{LoomisRevisited}.

%===================================================================================
%SUBSECTION
\subsection{Outline}

In Part~\ref{part:op} we characterise the operations that preserve integrability, and we provide a simple set of operations that generate them.
Specifically, we characterise the operations that preserve measurability, integrability, and integrability over finite measure spaces, respectively in Sections~\ref{s:meas},~\ref{s:lp}, and~\ref{s:lp finite}.
In Section~\ref{s:generation} we show that the operations $0$, $+$, $\lor$, $\lambda ({\;\cdot\;})$ (for each $\lambda \in \R$),  $\truncsup$ and  $\trunc$ generate the operations that preserve integrability, and that $0$, $+$, $\lor$, $\lambda ({\;\cdot\;})$ (for each $\lambda \in \R$),  $\truncsup$ and  $1$ generate the operations that preserve integrability over finite measure spaces.

In Part~\ref{part:variety} we prove that the categories of Dedekind $\sigma$-complete truncated Riesz spaces and Dedekind $\sigma$-complete Riesz spaces with weak unit are varieties generated by $\R$.
In more deatail,  in Section~\ref{s:truncsup} we define the operation $\truncsup$,  in Section~\ref{s:defn truncated} we define truncated lattice-ordered abelian groups, in Section~\ref{s:LS} we prove a version of the Loomis-Sikorski Theorem for truncated $\ell$-groups, in Section~\ref{s:truncated are variety} we show the category of Dedekind $\sigma$-complete truncated Riesz spaces to be generated by $\R$, in Section~\ref{s:LS weak} we prove a version of the Loomis-Sikorski Theorem for $\ell$-groups with weak unit, in Section~\ref{s:weak generated R} we show the category of Dedekind $\sigma$-complete Riesz spaces with weak unit to be generated by $\R$.

Finally, as an additional result, in the  Appendix  we provide an explicit characterisation of the operations that preserve $\infty$-integrability.

%===================================================================================
%SUBSECTION
\subsection*{Notation}

We let $\omega$ denote the set $\{0,1,2,\dots \}$.

%We set $\R^+ \df \{\lambda \in \R \mid \lambda \geq 0\}$.
%More generally, for a Riesz space $G$, we set $G^+ \df \{x \in G \mid x \geq 0\}$.
%
%When we regard $\R$ as a measurable space, we always do so with respect to the Borel $\sigma$-algebra, denoted by $\mathcal{B}_\R$.

%===================================================================================
%SUBSECTION
\subsection*{Acknowledgements}

The author is deeply grateful to his Ph.D.\ advisor prof.\ Vincenzo Marra for the many helpful discussions.

%&&&&&&&&&&&&&&&&&&&&&&&&&&&&&&&&&&&&&&&&&&&&&&&&&&&&&&&&&&&&&&&&&&&&&&&&&&&&&&&&&&&
%PART
\part{Operations that preserve integrability}\label{part:op}

%%%%%%%%%%%%%%%%%%%%%%%%%%%%%%%%%%%%%%%%%%%%%%%%%%%%%%%%%%%%%%%%%%%%%%%%%%%%%%%%%%%%
%SECTION
\section{Main results of Part 1}

In this section we state the main results of Part 1, together with the needed definitions.
The first two main results (Theorems~\ref{t:MAIN} and~\ref{t:MAIN finite} below) are a characterisation of the operations that preserve $p$-integrability over arbitrary and finite measure spaces, respectively.
The other two main results (Theorems~\ref{t:generation} and~\ref{t:generation finite}) provide a set of generators for these operations.
To state the theorems, we introduce a little piece of terminology.

For $I$ set, and $i\in I$, we let $\pi_i\colon \R^I\to \R$ denote the projection onto the $i$-th coordinate.
The \emph{cylinder $\sigma$-algebra on $\R^I$} (notation: $\Cyl\left(\R^I\right)$) is the smallest $\sigma$-algebra which makes each projection function $\pi_i\colon \R^I\to \R$ measurable.
If $\lvert I\rvert\leq \lvert \omega\rvert$, the cylinder $\sigma$-algebra on $\R^I$ coincides with the Borel $\sigma$-algebra (see \cite{Kallenberg}, Lemma~1.2).
\begin{theorem}\label{t:MAIN}
 	Let $I$ be a set, $\tau\colon \R^I\to \R$ and $p\in [1,+\infty)$.
	The following conditions are equivalent.
 	\begin{enumerate}
 		\item $\tau$ preserves $p$-integrability.
 		\item $\tau$ is $\Cyl\left(\R^I\right)$-measurable and there exist a finite subset of indices $J \seq I$ and nonnegative real numbers $(\lambda_j)_{j \in J}$ such that, for every $v\in \R^I$, we have 
 		\[
 			\lvert\tau(v)\rvert\leq \sum_{j \in J} \lambda_{j} \lvert v_j\rvert.
 		\]
 	\end{enumerate}
\end{theorem}
\begin{theorem}\label{t:MAIN finite}
  	Let $I$ be a set, $\tau\colon \R^I\to \R$ and $p\in [1,+\infty)$.
	The following conditions are equivalent.
  	\begin{enumerate}
  		\item $\tau$ preserves $p$-integrability over every finite measure space.
  		\item $\tau$ is $\Cyl\left(\R^I\right)$-measurable and there exist a finite subset of indices $J \seq I$ and nonnegative real numbers $(\lambda_j)_{j \in J}$ and $k$ such that, for every $v\in \R^I$, we have 
  			\[
  			\lvert\tau(v)\rvert\leq k + \sum_{j \in J} \lambda_{j} \lvert v_j\rvert.
  			\]
  	\end{enumerate}
\end{theorem}
%
%\noindent Note that it does not really matter if the indices $i_0, \dots, i_{n-1}$ in Theorems~\ref{t:MAIN} and~\ref{t:MAIN finite} are required to be distinct or not.
\noindent Theorems~\ref{t:MAIN} and~\ref{t:MAIN finite} show that the fact that an operation preserves $p$-integrability -- over arbitrary and finite measure spaces, respectively -- does not depend on the choice of $p$.
Hence, once Theorems~\ref{t:MAIN} and~\ref{t:MAIN finite} will be settled, we will simply say that the operation \emph{preserves integrability}.

The other two main results of Part I (Theorems~\ref{t:generation} and~\ref{t:generation finite} below) provide a set of generators for the operations that preserve integrability over arbitrary and finite measure spaces, respectively. To state the theorems, we start by defining, for any set $\mathcal{C}$ of operations $\tau\colon \R^{J_\tau}\to \R$, what we mean by \emph{operations generated by $\mathcal{C}$}.
Given two sets $\Omega$ and $I$, a subset $S\seq \R^\Omega$, and a function $\tau\colon \R^I\to \R$, we say that $S$ is \emph{closed under $\tau$} if, for  every family $(f_i)_{i\in I}$ of elements of $S$, we have that $\tau((f_i)_{i\in I})$ (which is the function from $\Omega$ to $\R$, which maps $\omega$ to $\tau((f_i(\omega))_{i\in I})$) belongs to $S$.
Consider a set $\mathcal{C}$ of functions $\tau\colon \R^{J_\tau}\to \R$, where the set $J_\tau$ depends on $\tau$.
We say that a function $f\colon \R^I\to \R$ is \emph{generated by $\mathcal{C}$} if $f$ belongs to the smallest subset of $\R^{\R^I}$ which contains, for each $i\in I$, the projection function $\pi_i\colon \R^I\to \R$, and which is closed under each element of  $\mathcal{C}$.
\begin{theorem}\label{t:generation}
  	For every set $I$, the operations $\R^I\to \R$ that preserve integrability are exactly those generated by the operations $0$, $+$, $\lor$, $\lambda ({\;\cdot\;})$ (for each $\lambda \in \R$),  $\truncsup$, and  $\trunc$.
\end{theorem}
\begin{theorem}\label{t:generation finite}
  	For every set $I$, the operations $\R^I\to \R$ that preserve integrability over every finite measure space are exactly those generated by the operations $0$, $+$, $\lor$, $\lambda ({\;\cdot\;})$ (for each $\lambda \in \R$),  $\truncsup$, and $1$.
\end{theorem}
The rest of Part~\ref{part:op} is devoted to a proof of Theorems~\ref{t:MAIN}-\ref{t:generation finite}.

%%%%%%%%%%%%%%%%%%%%%%%%%%%%%%%%%%%%%%%%%%%%%%%%%%%%%%%%%%%%%%%%%%%%%%%%%%%%%%%%%%%%
%SECTION
\section{Operations that preserve measurability}\label{s:meas}

In this section we study measurability, which is a necessary condition for integrability.
In particular, we characterise the operations that preserve measurability (Theorem~\ref{t:pres measurable}).
This result will be of use in the following sections as preservation of measurability is necessary to preservation of integrability (Lemma~\ref{l:lp implies measbl}).
Let us start by defining precisely what we mean by ``to preserve measurability''.
\begin{definition}
	Let $\tau\colon \R^I\to \R$ be a function.
	For $(\Omega,\mathcal{F})$ a measurable space, we say that $\tau$ \emph{preserves measurability over $(\Omega,\mathcal{F})$} if, for every family $(f_i)_{i\in I}$ of $\mathcal{F}$-measurable functions from $\Omega$ to $\R$, the function $\tau((f_i)_{i\in I})\colon \Omega\to \R$ is also $\mathcal{F}$-measurable.
	We say that $\tau$ \emph{preserves measurability} if $\tau$ preserves measurability over every measurable space.\qed
\end{definition}
When we regard $\R$ as a measurable space, we always do so with respect to the Borel $\sigma$-algebra, denoted by $\mathcal{B}_\R$.
\begin{lemma}\label{l:product}
	Let $(\Omega,\mathcal{F})$ be a measurable space, $I$ a set and $f\colon \Omega\to \R^I$ a function.
	Then $f$ is $\mathcal{F}$-$\Cyl\left(\R^I\right)$-measurable if, and only if, for every $i\in I$ the function $\pi_i\circ f \colon \Omega\to\R$ is $\mathcal{F}$-$\mathcal{B}_\R$-measurable.
\end{lemma}
\begin{proof}
	See \cite{Borelsets}, Theorem~3.1.29.(ii).
\end{proof}
Now we can obtain a characterisation of the operations that preserve measurability.
\begin{theorem}\label{t:pres measurable}
	Let $I$ be a set and let $\tau\colon \R^I\to \R$ be  a function.
	The following are equivalent.
	\begin{enumerate}
		\item $\tau$ preserves measurability.
		\item $\tau$ preserves measurability over $\left(\R^I,\Cyl\left(\R^I\right)\right)$.
		\item  $\tau$ is $\Cyl\left(\R^I\right)$-measurable.
	\end{enumerate}
\end{theorem}
\begin{proof}
	$[(1)\Rightarrow(2)]$
	Trivial.
 	
	$[(2)\Rightarrow(3)]$ For every $i\in I$, $\pi_i\colon \R^I\to \R$ is $\Cyl\left(\R^I\right)$-measurable.
	Since $\tau$ preserves measurability,  $\tau((\pi_i)_{i\in I})$ is $\Cyl\left(\R^I\right)$-measurable.
	Since $(\pi_i)_{i\in I}\colon \R^I\to \R^I$ is the identity, $\tau((\pi_i)_{i\in I})=\tau \circ (\pi_i)_{i\in I}=\tau$ is $\Cyl\left(\R^I\right)$-measurable.
 	
	$[(3)\Rightarrow(1)]$ Let us consider a measurable space $(\Omega,\mathcal{F})$ and a family $(f_i)_{i\in I}$ of measurable functions $f_i\colon \Omega\to \R$.
	Consider the function $(f_i)_{i\in I}\colon \Omega\to \R^I$, $x\mapsto (f_i(x))_{i\in I}$.
	We have $\pi_i\circ (f_i)_{i\in I}=f_i$, therefore $\pi_i\circ (f_i)_{i\in I}$ is measurable for every $i\in I$.
	Thus, by Lemma~\ref{l:product}, $(f_i)_{i\in I}$ is measurable.
	Thus $\tau((f_i)_{i\in I})=\tau\circ (f_i)_{i\in I}$ is measurable, because it is a composition of measurable functions.
\end{proof}
%

%===================================================================================
%SUBSECTION
\subsection{The operations that preserve measurability depend on countably many coordinates}

A fact that will be of use in the following sections is that the operations that preserve measurability depend on countably many coordinates.
This we show in Corollary~\ref{c:meas then count} below.
Let us start by recalling what is meant with ``to depend on countably many coordinates''.
\begin{definition}
	Given a set $I$.
	\begin{enumerate}
		\item Let $S\seq \R^I$.
		For $J\seq I$, we say that $S$ \emph{depends only on $J$} if, given any $x,y\in \R^I$ such that  $x_j=y_j$ for all $j\in J$, we have $x\in S\Leftrightarrow y\in S$.
		We say that $S$ \emph{depends on countably many coordinates} if there exists a countable subset $J\seq I$ such that $S$ depends only on $J$.
		\item Let $\tau\colon \R^I\to \R$ be a function.
		For $J\seq I$, we say that $\tau$ depends only on $J$ if, given any $x,y\in \R^I$ such that $x_j=y_j$ for all $j\in J$, we have $\tau(x)=\tau(y)$.
		We say that $\tau$ depends on countably many coordinates if there exists a countable subset $J\seq I$ such that $\tau$ depends only on $J$.\qed
	\end{enumerate}
\end{definition}
We believe that the following proposition is folklore, but we were not able to locate an appropriate reference.
\begin{proposition}\label{p:depends on countably}
	If $\tau\colon \R^I\to \R$ is $\Cyl\left(\R^I\right)$-measurable, then $\tau$ depends on countably many coordinates.
\end{proposition}
\begin{proof}
	First, every element of $\Cyl\left(\R^I\right)$ depends on countably many coordinates: indeed, the set of elements of $\Cyl\left(\R^I\right)$ which depend on countably many coordinates is a $\sigma$-subalgebra of $\Cyl\left(\R^I\right)$ which makes the projection functions measurable (see also  254M(c) in \cite{Fremlin}).
	Second, let $ \tau \colon \R^I\to \R$ be $\Cyl\left(\R^I\right)$-measurable.
	The idea that we will use is that $\tau$ is determined by the family $(\tau^{-1}((a,+\infty)))_{a\in \Q}$.
	For every $a\in \Q$, there exists a countable subset $J\seq I$ such that the measurable set $\tau^{-1}((a,+\infty))$ depends only on $J_a$.
	Then $J\df\bigcup_{a\in \Q} J_a$ has the property that, for each $b\in \Q$, $\tau^{-1}((b,+\infty))$ depends only on $J$.
	We claim that $\tau$ depends only on $J$.
	Let $x,y\in \R^I$ be such that $x_j=y_j$ for every $j\in J$.
	We shall prove $\tau(x)=\tau(y)$.
	Suppose $\tau(x)\neq \tau(y)$.
	Without loss of generality, $\tau(x)<\tau(y)$.
	Let $a\in \Q$ be such that $\tau(x)<a<\tau(y)$.
	Then $x\notin \tau^{-1}((a,+\infty))$ and $y\in \tau^{-1}((a,+\infty))$.
	This implies that it is not true that $\tau^{-1}((a,+\infty))$ depends only on $J$.
\end{proof}
\begin{corollary}\label{c:meas then count}
	Let $I$ be a set and $\tau\colon \R^I\to \R$ be a function.
	If $\tau$ preserves measurability, then $\tau$ depends on countably many coordinates.
\end{corollary}
\begin{proof}
	If $\tau$ preserves measurability, then $\tau$ is $\Cyl\left(\R^I\right)$-measurable by Theorem~\ref{t:pres measurable}.
	By Proposition~\ref{p:depends on countably}, $\tau$ depends on countably many coordinates.
\end{proof}
%

%===================================================================================
%SUBSECTION
\subsection{The case of uncountable Polish spaces}

The remaining results in this section are not used in the proofs of our main results.

One may think that, for an operation $\tau \colon\R^I\to \R$, the condition ``$\tau$ preserve measurability over \emph{every} measurable space'' is too strong because we may not be interested in all measurable spaces.
However, Proposition~\ref{p:R is quite important} shows that this condition is equivalent to ``$\tau$ preserve measurability over $(\R,\mathcal{B}_\R)$'' (if $\tau$ has countable arity).
\begin{proposition}\label{p:R is quite important}
 	For a set $I$ such that $\lvert I \rvert\leq \lvert \omega\rvert$ and a function $\tau\colon \R^I\to \R$, $\tau$ preserves measurability if, and only if, $\tau$ preserves measurability on $(\R,\mathcal{B}_\R)$.
\end{proposition}
\begin{proof}
	If $I=\emptyset$, then $\tau$ is a constant function.
	Hence $\tau$ preserves measurability over every measurable space.
	Let us consider the case $I\neq \emptyset$.
	By Theorem~\ref{t:pres measurable}, $\tau$ preserves measurability if, and only if,   $\tau$ preserves measurability over $(\R,\Cyl\left(\R^I\right))$.
	Since $\R^I$ and $\R$ are uncountable Polish spaces with Borel $\sigma$-algebras $\Cyl\left(\R^I\right)$ and $\mathcal{B}_\R$ respectively, $(\R^I,\Cyl\left(\R^I\right))$ and $(\R,\mathcal{B}_\R)$ are isomorphic measurable spaces (see \cite{Borelsets}, Theorem~3.3.13).
	(Recall that an isomorphism of measurable spaces $(\Omega,\mathcal{F})$ and $(\Omega',\mathcal{F}')$ is a bijective measurable function $f\colon \Omega\to \Omega'$ such that its inverse is measurable.)
\end{proof}
\begin{remark}\label{r:polish is quite important}
	In Proposition~\ref{p:R is quite important} above, one may replace the measurable space $(\R,\mathcal{B}_\R)$ by any of its isomorphic copies.
	In particular, one may replace it with the measurable space given by any uncountable Polish space endowed with its Borel $\sigma$-algebra (see \cite{Borelsets}, Chapter 3).
\end{remark}
%

%%%%%%%%%%%%%%%%%%%%%%%%%%%%%%%%%%%%%%%%%%%%%%%%%%%%%%%%%%%%%%%%%%%%%%%%%%%%%%%%%%%%
%SECTION
\section{Operations that preserve integrability}\label{s:lp}

The goal of this section is to prove Theorem~\ref{t:MAIN}, i.e.\ to characterise the operations that preserve $p$-integrability.

\begin{remark}\label{r:null-meas}
	Let $(\Omega, \mathcal{F})$ be a measurable space, and let $\mu_0$ be the null-measure on $(\Omega, \mathcal{F})$: for each $A\in \mathcal{F}$, $\mu_0(A)=0$.
	Then $\mathcal{L}^p(\mu_0)$ is the set of $\mathcal{F}$-measurable functions from $\Omega$ to $\R$.
	Hence, preservation of  $p$-integrability over $(\Omega, \mathcal{F},\mu_0)$ is equivalent to preservation of measurability over $(\Omega, \mathcal{F})$.
\end{remark}
An immediate consequence of Remark~\ref{r:null-meas} is the following lemma.
\begin{lemma}\label{l:lp implies measbl}
	Let $I$ be a set, $\tau\colon \R^I\to \R$ and $p\in [1,+\infty)$.
	If $\tau$ preserves $p$-integrability, then $\tau$ preserves measurability.
\end{lemma}
\begin{lemma}\label{l:implies lp}
	Let $(\Omega,\mathcal{F},\mu)$ be a measure space, and let $f,g\colon \Omega\to \R$ be functions, and let $\lambda \in \R$.
	Then the following properties hold.
	\begin{enumerate}
		\item If $f\in \lpmu$, then $\lvert f\rvert \in \lpmu$.
		\item If $f \in \lpmu$, then $\lambda f \in \lpmu$.
		\item If $f,g\in \lpmu$, then $f+g\in \lpmu$.
		\item If $g\in \lpmu, \lvert f\rvert\leq \lvert g\rvert$ and $f$ is $\mathcal{F}$-measurable, then  $f\in \lpmu$.
	\end{enumerate}
\end{lemma}
\begin{proof}
	(1) is immediate by definition of $\lpmu$, (2) follows from linearity of the integration operator, (4) follows from the monotonicity of the integration operator, while (3) follows from the Minkowski inequality (see \cite{Rudin}, Theorem~3.5): 
	\[
		\left(\int_{\Omega} \lvert f + g \rvert^p \,d\mu \right)^{\frac{1}{p}} \leq 
		\left(\int_{\Omega} (\lvert f \rvert + \lvert g \rvert)^p \,d\mu\right)^{\frac{1}{p}} \stackrel{\text{Mink.}}{\leq}
		\left(\int_{\Omega} \lvert f \rvert^p \,d\mu \right)^{\frac{1}{p}} + \left(\int_{\Omega} \lvert g \rvert^p \,d\mu\right)^{\frac{1}{p}}.
	\]
\end{proof}
The next lemma settles the easiest direction of the characterisation of operations that preserve $p$-integrability, i.e.\ the implication (2) $\Rightarrow$ (1) in Theorem \ref{t:MAIN}.

\begin{lemma}\label{l:easy implication}
	Let $(\Omega, \mathcal{F},\mu)$ be a measure space, $I$ a set, $\tau\colon \R^I\to \R$ an operation that preserves measurability over $(\Omega, \mathcal{F})$ and $p\in [1,+\infty)$.
	If there exist a finite subset of indices $J \seq I$ and nonnegative real numbers $(\lambda_j)_{j \in J}$ such that, for every $v\in \R^I$, we have $\lvert\tau(v)\rvert\leq \sum_{j \in J} \lambda_{j} \lvert v_{j}\rvert$,	then $\tau$ preserves $p$-integrability over $(\Omega, \mathcal{F},\mu)$.
\end{lemma}
\begin{proof}
	Let $(f_i)_{i\in I}$ be a family in $\lpmu$; since $\tau$ preserves measurability over $(\Omega, \mathcal{F})$,  $\tau((f_i)_{i\in I})$ is $\mathcal{F}$-measurable.
	For each $x\in \Omega$, $\lvert \tau((f_i(x))_{i\in I}) \rvert \leq \sum_{j \in J}\lambda_j \lvert f_{j}(x)\rvert$.
	 Thus $\lvert \tau((f_i)_{i\in I})\rvert\leq \sum_{j \in J} \lambda_j \lvert f_{j}\rvert$.
	Hence, by Lemma~\ref{l:implies lp}, $\tau((f_i)_{i\in I})\in \lpmu$.
\end{proof}
\noindent This shows that the condition $\lvert\tau(v)\rvert\leq \sum_{j \in J} \lambda_{j} \lvert v_{j}\rvert$ is sufficient for preservation of $p$-integrability.
We are left to prove the converse direction: when $\tau$ does not satisfy this condition, there exists a measure space over which $\tau$ does not preserve $p$-integrability.
As we shall see, at least when the arity of $\tau$ is countable, this space can always be taken to be $(\R, \mathcal{B}_{\R}, \Leb)$ where $\Leb$ is the restriction to $\mathcal{B}_{\R}$ of the Lebesgue measure, and this happens because $(\R, \mathcal{B}_{\R}, \Leb)$ is what we call a partitionable measure space.
\begin{definition}\label{d:partitionable}
	A measure space $(\Omega,\mathcal{F},\mu)$ is called \emph{partitionable} if, for every sequence $(a_n)_{n\in \omega}$ of elements of $\R^+$, there exists a sequence $(A_n)_{n\in \omega}$ of disjoint elements of $\mathcal{F}$ such that $\mu(A_n)=a_n$. \qed
\end{definition}
\begin{remark}\label{r:R is partbl}
	The measure space $(\R,\mathcal{B}_\R,\Leb)$ is partitionable.
\end{remark}
The role of partitionable measure spaces is clarified by the following result.
\begin{lemma}\label{l:l2}
	Let $(\Omega, \mathcal{F},\mu)$ be a measure space, let $p\in [1,+\infty)$, let $I$ be a set and let $\tau\colon \R^I\to \R$ be a function.
	Suppose $\lvert I\rvert \leq \lvert \omega\rvert$ and suppose $(\Omega, \mathcal{F},\mu)$ is partitionable.
	If $\tau$ preserves $p$-integrability over $(\Omega, \mathcal{F},\mu)$, then there exist a finite subset of indices $J \seq I$ and nonnegative real numbers $(\lambda_j)_{j \in J}$ such that, for every $v\in \R^I$, we have 
	\[
		\lvert\tau(v)\rvert\leq \sum_{j \in J} \lambda_{j} \lvert v_{j}\rvert.
	\]
\end{lemma}
\begin{proof}
	We give the proof for $I = \omega$.
	The case $\lvert I\rvert <\lvert \omega\rvert$ relies on an analogous argument.
	
	We suppose, contrapositively, that, for every finite subset of indices $J \seq I$ and every $J$-tuple $(\lambda_j)_{j \in J}$ of nonnegative real numbers, there exists $v\in \R^I$ such that $\lvert \tau(v) \rvert > \sum_{j \in J} \lambda_j \lvert v_{j}\rvert$; we shall prove that $\tau$ does not preserve $p$-integrability.
	For each $n\in \omega$, we let $v^n$ be an element of $\R^I$ such that $\lvert \tau(v^n) \rvert > \sum_{j = 0}^{n - 1} 2^{\frac{n}{p}} \lvert v^n_{j}\rvert$.
	Set $C\df\Omega\setminus \bigcup_{n\in \omega}A_n$.
	For each $i \in \omega$, we set
	\begin{align*}
		f_i\colon	\Omega	& \longrightarrow	\R\\
					x		& \longmapsto
												\begin{cases}
													v^n_i	& \text{if } x \in A_n; \\
													0		& \text{if } x \in C.
												\end{cases}
	\end{align*}
	Let $(A_n)_{n\in \omega}$ be a sequence of disjoint elements of $\mathcal{F}$ such that $\mu(A_n)= \frac{1}{\lvert\tau(v^n)\rvert^p}$; one such sequence exists because $(\Omega, \mathcal{F}, \mu)$ is partitionable.
	Then,
	\begin{equation} \label{e:integral-is-infty}
		\begin{split}
			\int_{\Omega} \lvert\tau( (f_i)_{i\in \omega})\rvert^p\,d\mu
				& =		\int_{C} \lvert\tau( (f_i)_{i\in \omega})\rvert^p\,d\mu + \sum_{n\in \omega} \int_{A_n} \lvert\tau( (f_i)_{i\in \omega})\rvert^p\,d\mu \geq \\
				& \geq	\sum_{n\in \omega} \lvert \tau((v_i^n)_{i \in \omega})\rvert^p \mu(A_n) = \\
				& =		\sum_{n\in \omega} \lvert \tau(v^n) \rvert^p \frac{1}{\lvert \tau(v^n) \rvert^p} = \\
				& =		\sum_{n\in \omega} 1 = \\
				& =		\infty.
		\end{split}
	\end{equation}	
	The following chain of inequalities holds.
	\begin{equation} \label{e:integral-is-finite}
		\begin{aligned}
			\int_{\Omega} \lvert f_i\rvert^p\,d\mu	
				& =		\sum_{n\in \omega}\lvert v_i^n\rvert^p\mu(A_n) = 																									& \ \ \ \ \ \ \\
				& =		\sum_{n\in \omega}\lvert v^n_{i} \rvert^p \frac{1}{\lvert\tau(v^n)\rvert^p} \leq 																	& \ \ \ \ \ \ \\
				& \leq	M + \sum_{n > i, v_{i}^n \neq 0} \lvert v^n_{i} \rvert^p \frac{1}{\lvert\tau(v^n)\rvert^p} \leq 													& \ \ \ \ \ \ \text{(for some $M\in \R^+$)} \\
				& \leq	M + \sum_{n > i, v_{i}^n \neq 0} \lvert v^n_{i} \rvert^p \frac{1}{(\sum_{j = 0}^{n - 1} 2^{\frac{n}{p}} \lvert v^n_{j}\rvert)^p} \leq	& \ \ \ \ \ \ \\
				& \leq	M + \sum_{n > i, v_{i}^n \neq 0} \lvert v^n_{i} \rvert^p \frac{1}{\left(2^{\frac{n}{p}} \lvert v^n_{i} \rvert\right)^p} \leq 			& \ \ \ \ \ \ \\
				& \le	M + \sum_{n > i, v_{i}^n \neq 0} \frac{1}{2^{n}} < 																									& \ \ \ \ \ \ \\
				& <		\infty. 																																			& \ \ \ \ \ \ 
		\end{aligned}
	\end{equation}
	The first inequality holds for some $M\in \R^+$ because with the condition $n>i$ we ignore finitely many terms of the series, while with the condition $v_{i}^n\neq 0$ we ignore some null terms.
	The third inequality holds because $n>i\Rightarrow i\in \{0,\dots,{n-1} \}$.
	
	From eqs.\ \eqref{e:integral-is-infty} and \eqref{e:integral-is-finite} we conclude that $\tau$ does not preserve $p$-integrability.
\end{proof}
\begin{lemma}\label{l:one measure integrability}
	If $I$ is a set, $\tau\colon \R^I\to \R$ a function, $p\in [1,+\infty)$ and $(\Omega,\mathcal{F},\mu)$ a partitionable measure space, then the following conditions are equivalent.
	\begin{enumerate}
		\item $\tau$ preserves $p$-integrability.
		\item $\tau$ preserves measurability, and $\tau$ preserves $p$-integrability over $(\Omega,\mathcal{F},\mu)$.
		\item $\tau$ is $\Cyl\left(\R^I\right)$-measurable and there exist a finite subset of indices $J \seq I$ and nonnegative real numbers $(\lambda_j)_{j \in J}$ such that, for every $v\in \R^I$, we have $\lvert\tau(v)\rvert\leq \sum_{j \in J} \lambda_{j} \lvert v_{j}\rvert$.
	\end{enumerate}
\end{lemma}
\begin{proof}
	$[(1)\Rightarrow (2)]$ If $\tau$ preserves $p$-integrability, then, by Lemma~\ref{l:lp implies measbl}, $\tau$ preserves measurability.
	Trivially, $\tau$ preserves $p$-integrability over $(\Omega,\mathcal{F},\mu)$.
	
	$[(2)\Rightarrow (3)]$ If $\tau$ preserves measurability, then, by Theorem~\ref{t:pres measurable}, $\tau$ is $\Cyl\left(\R^I\right)$-measurable.
	By Proposition~\ref{p:depends on countably}, $\tau$ depends on countably many coordinates, hence Lemma~\ref{l:l2} applies and the proof of the implication is complete.

	$[(3)\Rightarrow (1)]$ By Theorem~\ref{t:pres measurable}, $\tau$ preserves measurability.
	By Lemma~\ref{l:easy implication}, the thesis is proved.
\end{proof}
\begin{proof}[Proof of Theorem~\ref{t:MAIN}]
	There exist partitionable measure spaces, see e.g.\ Remark~\ref{r:R is partbl}.
	Theorem~\ref{t:MAIN} is the equivalence $(1)\Leftrightarrow(3)$ in Lemma~\ref{l:one measure integrability}.
\end{proof}
%

%===================================================================================
%SUBSECTION
\subsection{Examples}

\begin{example}
	Let $n\in\omega$ and $\tau\colon \R^n\to \R$.
	Then $\tau$ preserves $p$-integrability if, and only if, $\tau$ is Borel measurable and there exist $\lambda_0, \dots, \lambda_{n-1}\in \R^+$ such that, for every $x\in \R^n$, we have
	\[
		\lvert \tau(x) \rvert \leq \sum_{j = 0}^{n-1} \lambda_i \lvert x_i \rvert.
	\]
\end{example}
\begin{example}
	A function $\tau\colon \R^\omega\to \R$ preserves $p$-integrability if, and only if, $\tau$ is Borel measurable and there exist a finite subset of indices $J \seq \omega$ and nonnegative real numbers $(\lambda_j)_{j \in J}$ and $k$ such that, for every $v\in \R^I$, we have 
	\[
		\lvert\tau(v)\rvert \leq k + \sum_{j \in J} \lambda_{j} \lvert v_j\rvert.
	\]
\end{example}
%

%===================================================================================
%SUBSECTION
\subsection{The case of $(\R,\mathcal{B}_\R, \Leb)$ and the discrete case}

The remaining results in this section are not used in the proofs of our main results.

One may think that, for an operation $\tau \colon\R^I\to \R$, the condition ``$\tau$ preserve $p$-integrability over \emph{every} measure space'' is too strong because we may not be interested in all measure spaces.
However, Proposition~\ref{p:R is enough} shows that this condition is equivalent to ``$\tau$ preserve $p$-integrability over $(\R,\mathcal{B}_\R, \Leb)$'' (if $\tau$ has countable arity), and  Proposition~\ref{p:discrete is enough} provides an analogous result for a discrete measure space.
\begin{proposition}\label{p:R is enough}
	Let $I$ be a set, $\tau\colon \R^I\to \R$, with $\lvert I\rvert\leq \lvert\omega\rvert$, and $p\in [1,+\infty)$.
	Then $\tau$ preserves $p$-integrability
	if, and only if, $\tau$ preserves $p$-integrability over $(\R,\mathcal{B}_\R,\Leb)$.
\end{proposition}
\begin{proof}
	Trivially, if $\tau$ preserves $p$-integrability, then $\tau$ preserves $p$-integrability over $(\R,\mathcal{B}_\R,\Leb)$.
	For the converse, by Proposition~\ref{p:R is quite important}, if $\tau$ preserves $p$-integrability over $(\R,\mathcal{B}_\R,\Leb)$ then $\tau$ preserves measurability.
	By Remark~\ref{r:R is partbl}, $(\R,\mathcal{B}_\R,\Leb)$ is partitionable.
	An application of $(2)\Rightarrow(1)$ in Lemma~\ref{l:one measure integrability} concludes the proof.
\end{proof}
We next provide an analogue of Proposition~\ref{p:R is enough} for a discrete measure space.
We denote by $\mathcal{P}(X)$ the power set of a set $X$.
\begin{lemma}\label{l:exists measure}
	There exists a measure $\mu$ on $(\omega,\mathcal{P}(\omega))$ such that $(\omega,\mathcal{P}(\omega),\mu)$ is partitionable.
\end{lemma}
\begin{proof}
	We define a measure $\mu$ on $(\omega\times\Z,\mathcal{P}(\omega\times\Z))$, by setting $\mu(\{(n,z)\})=2^z$.
	For every $n\in \omega$, there exists $K_n\seq \Z$ such that $a_n=\sum_{z\in K_n}2^z$.
	Set $A_n\df\{(n,z)\mid z\in K_n \}$.
	Then $\mu(A_n)=\sum_{z\in K_n}\mu(\{(n,z)\})=\sum_{z\in K_n}2^z=a_n$.
	Moreover, for any pair of distinct $n, m\in \omega$, the sets $A_n$ and $A_m$ are disjoint.
	The set $\omega\times\Z$ is countably infinite, hence $(\omega\times\Z,\mathcal{P}(\omega\times \Z))$ and $(\omega,\mathcal{P}(\omega))$ are isomorphic measurable spaces, which concludes the proof.
\end{proof}
\begin{proposition}\label{p:discrete is enough}
	There exists a measure $\mu$ on $(\omega,\mathcal{P}(\omega))$ such that, for every set $I$, every function $\tau\colon \R^I\to \R$ and every $p\in [1,+\infty)$, $\tau$ preserves $p$-integrability if, and only if, $\tau$ preserves measurability and $\tau$ preserves $p$-integrability over $(\omega,\mathcal{P}(\omega),\mu)$.
\end{proposition}
\begin{proof}
	By Lemma~\ref{l:exists measure}, there exists a measure $\mu$ on $(\omega,\mathcal{P}(\omega))$ such that $(\omega,\mathcal{P}(\omega),\mu)$ is partitionable.
	The thesis follows from $(1)\Leftrightarrow (2)$ in Lemma~\ref{l:one measure integrability}.
\end{proof}
%

%%%%%%%%%%%%%%%%%%%%%%%%%%%%%%%%%%%%%%%%%%%%%%%%%%%%%%%%%%%%%%%%%%%%%%%%%%%%%%%%%%%%
%SECTION
\section{Operations that preserve integrability over finite measure spaces}\label{s:lp finite}

The goal of this section is to prove Theorem~\ref{t:MAIN finite}, i.e.\ to characterise the operations that preserve $p$-integrability over finite measure spaces.
We follow the same strategy of Section \ref{s:lp}, with the appropriate adjustments.
\begin{lemma}\label{l:lp implies measbl finite}
	Let $I$ be a set, $\tau\colon \R^I\to \R$ and $p\in [1,+\infty)$.
	If $\tau$ preserves $p$-integrability over every finite measure space, then $\tau$ preserves measurability.
\end{lemma}
\begin{proof}
	By Remark~\ref{r:null-meas}.
\end{proof}
\begin{lemma}\label{l:easy implication-finite}
	Let $(\Omega, \mathcal{F},\mu)$ be a finite measure space, $I$ a set, $\tau\colon \R^I\to \R$ an operation that preserves measurability over $(\Omega, \mathcal{F})$ and $p\in [1,+\infty)$.
	If there exist a finite subset of indices $J \seq I$ and nonnegative real numbers $(\lambda_j)_{j \in J}$ and $k$ such that, for every $v\in \R^I$, we have 
	$\lvert\tau(v)\rvert\leq k + \sum_{j \in J} \lambda_{j} \lvert v_j\rvert$, then $\tau$ preserves $p$-integrability over $(\Omega, \mathcal{F},\mu)$.
\end{lemma}
\begin{proof}
	Let $(f_i)_{i\in I}$ be a family in $\lpmu$; since $\tau$ preserves measurability over $(\Omega, \mathcal{F})$, we have that $\tau((f_i)_{i\in I})$ is $\mathcal{F}$-measurable.
	For each $x\in \Omega$, we have $\lvert\tau((f_i(x))_{i\in I})\rvert\leq k + \sum_{j \in J} \lambda_j \lvert f_{j}(x)\rvert$.
	Thus $\lvert \tau((f_i)_{i\in I})\rvert \leq k + \sum_{j \in J} \lambda_j \lvert f_{j}\rvert$.
	Note that the function $k\colon \Omega\to \R, x\mapsto k$ belongs to $\lpmu$, because $\mu$ is finite.
	Hence, by Lemma~\ref{l:implies lp},  $\tau((f_i)_{i\in I})\in \lpmu$.
\end{proof}
It is not difficult to see that no finite measure space is partitionable:
thus we replace the concept of partitionability with a slightly different one.
\begin{definition}\label{d:partitionable-finite}
	A measure space $(\Omega,\mathcal{F},\mu)$ is called \emph{conditionally partitionable} if there exists a sequence $(b_n)_{n \in \omega}$ of strictly positive real numbers such that, for every sequence $(a_n)_{n \in \omega}$ of elements of $\R^+$ satisfying $a_n \leq b_n$ for every $n \in \omega$, there exists a sequence $(A_n)_{n\in \omega}$ of disjoint elements of $\mathcal{F}$ such that $\mu(A_n)=a_n$.
	\qed
\end{definition}
\begin{remark}\label{r:[0,1] is partbl}
	The measure space $([0,1],\mathcal{B}_{[0,1]},\Leb)$, where $\Leb$ is the Lebesgue measure, is conditionally partitionable (take $b_n = \frac{1}{2^{n+1}}$).
\end{remark}

\begin{lemma}\label{l:l2-finite}
	Let $(\Omega, \mathcal{F},\mu)$ be a measure space, let $p\in [1,+\infty)$, let $I$ be a set and let $\tau\colon \R^I\to \R$ be a function.
	Suppose that $\lvert I\rvert \leq \lvert \omega\rvert$ and that $(\Omega, \mathcal{F},\mu)$ is conditionally partitionable.
	If $\tau$ preserves $p$-integrability over $(\Omega, \mathcal{F},\mu)$, then there exist a finite subset of indices $J \seq I$ and nonnegative real numbers $(\lambda_j)_{j \in J}$ and $k$ such that, for every $v\in \R^I$, we have 
	\[
		\lvert\tau(v)\rvert\leq k + \sum_{j \in J} \lambda_{j} \lvert v_{j}\rvert.
	\]
\end{lemma}
\begin{proof}
	We give the proof for $I = \omega$.
	The case $\lvert I\rvert <\lvert \omega\rvert$ relies on an analogous argument.
	
	We suppose, contrapositively, that, for every finite subset of indices $J \seq I$, every $J$-tuple $(\lambda_j)_{j \in J}$ of nonnegative real numbers and every $k \in \R^+$, there exists $v\in \R^I$ such that $\lvert\tau(v)\rvert > k + \sum_{j \in J} \lambda_{j} \lvert v_{j}\rvert$; we shall prove that $\tau$ does not preserve $p$-integrability.
	Since $(\Omega, \mathcal{F}, \mu)$ is conditionally partitionable, there exists a sequence $(b_n)_{n \in \omega}$ of strictly positive real numbers such that, for every sequence $(a_n)_{n \in \omega}$ of elements of $\R^+$ satisfying $a_n \leq b_n$ for every $n \in \omega$, there exists a sequence $(A_n)_{n\in \omega}$ of disjoint elements of $\mathcal{F}$ such that $\mu(A_n)=a_n$. 
	
	For each $n\in \omega$, we let $v^n$ be an element of $\R^I$ such that $\lvert \tau(v^n) \rvert > \left(\frac{1}{b_n}\right)^{\frac{1}{p}} + \sum_{j = 0}^{n - 1} 2^{\frac{n}{p}} \lvert v^n_{j}\rvert$.
	Then we have
	\[
		\frac{1}{\lvert \tau(v^n) \rvert^p} < \frac{1}{\left(\left(\frac{1}{b_n}\right)^{\frac{1}{p}} + \sum_{j = 0}^{n - 1} 2^{\frac{n}{p}} \lvert v^n_{j}\rvert\right)^p} \leq \frac{1}{\left(\left(\frac{1}{b_n}\right)^{\frac{1}{p}}\right)^p} = b_n.
	\]
	Therefore, there exists a sequence $(A_n)_{n\in \omega}$ of disjoint elements of $\mathcal{F}$ such that $\mu(A_n)=\frac{1}{\lvert \tau(v^n) \rvert^p}$.
	Since $\lvert \tau(v^n) \rvert > \left(\frac{1}{b_n}\right)^{\frac{1}{p}} + \sum_{j = 0}^{n - 1} 2^{\frac{n}{p}} \lvert v^n_{j}\rvert > \sum_{j = 0}^{n - 1} 2^{\frac{n}{p}} \lvert v^n_{j}\rvert$, the remaining part of the proof runs as for Lemma \ref{l:one measure integrability}.
\end{proof}
\begin{lemma}\label{l:one measure integrability finite}
	Let $I$ be a set, $\tau\colon \R^I\to \R$ a function, $p\in [1,+\infty)$ and $(\Omega,\mathcal{F},\mu)$ a conditionally partitionable finite measure space.
	The following conditions are equivalent.
	\begin{enumerate}
		\item $\tau$ preserves $p$-integrability over every finite measure space.
		\item $\tau$ preserves measurability, and $\tau$ preserves $p$-integrability over $(\Omega,\mathcal{F},\mu)$.
		\item $\tau$ is $\Cyl\left(\R^I\right)$-measurable and there exist a finite subset of indices $J \seq I$ and nonnegative real numbers $(\lambda_j)_{j \in J}$ and $k$ such that, for every $v\in \R^I$, we have $\lvert\tau(v)\rvert\leq k + \sum_{j \in J} \lambda_{j} \lvert v_j\rvert$.
	\end{enumerate}
\end{lemma}
\begin{proof}	
	$[(1)\Rightarrow (2)]$ If $\tau$ preserves $p$-integrability over every finite measure space, then, by Lemma~\ref{l:lp implies measbl finite}, $\tau$ preserves measurability.
	Trivially, $\tau$ preserves $p$-integrability over $(\Omega,\mathcal{F},\mu)$.
	
	$[(2)\Rightarrow (3)]$ If $\tau$ preserves measurability, then, by Theorem~\ref{t:pres measurable}, $\tau$ is $\Cyl\left(\R^I\right)$-measurable.
	By Proposition~\ref{p:depends on countably}, $\tau$ depends on countably many coordinates, hence Lemma~\ref{l:l2-finite} applies and the proof of the implication is complete.
	
	$[(3)\Rightarrow (1)]$ By Theorem~\ref{t:pres measurable}, $\tau$ preserves measurability.
	By Lemma~\ref{l:easy implication-finite}, the thesis is proved.
\end{proof}
\begin{proof}[Proof of Theorem~\ref{t:MAIN finite}]
	There exist conditionally partitionable finite measure spaces, see e.g.\ Remark~\ref{r:[0,1] is partbl}.
	Theorem~\ref{t:MAIN finite} is the equivalence $(1)\Leftrightarrow(3)$ in Lemma~\ref{l:one measure integrability finite}.
\end{proof}
%

%===================================================================================
%SUBSECTION
\subsection{Examples}

\begin{example}
	Let $n\in\omega$ and $\tau\colon \R^n\to \R$.
	Then $\tau$ preserves $p$-integrability over every finite measure space if, and only if, $\tau$ is Borel measurable and there exist $\lambda_0, \dots, \lambda_{n-1}, k \in \R^+$ such that, for every $x\in \R^n$, we have
	\[
		\lvert \tau(x)\rvert \leq k + \sum_{j = 0}^{n - 1}\lambda_j\lvert x_j\rvert.
	\]
\end{example}
\begin{example}
	A function $\tau\colon \R^\omega\to \R$ preserves $p$-integrability over every finite measure space if, and only if, $\tau$ is Borel measurable and there exist a finite subset of indices $J \seq \omega$ and nonnegative real numbers $(\lambda_j)_{j \in J}$ and $k$ such that, for every $v\in \R^I$, we have 
	\[
		\lvert\tau(v)\rvert\leq k + \sum_{j \in J} \lambda_{j} \lvert v_j\rvert.
	\]
\end{example}
%

%===================================================================================
%SUBSECTION
\subsection{The case of $([0,1],\mathcal{B}_{[0,1]}, \Leb)$ and the discrete case}

The remaining results in this section are not used in the proofs of our main results.

One may think that, for an operation $\tau \colon\R^I\to \R$, the condition ``$\tau$ preserve $p$-integrability over \emph{every} finite measure space'' is too strong because we may not be interested in all finite measure spaces.
However, Proposition~\ref{p:[0,1] is enough} shows that this condition is equivalent to ``$\tau$ preserve $p$-integrability over $([0,1],\mathcal{B}_{[0,1]},\Leb)$'' (at least when $\tau$ has countable arity), and  Proposition~\ref{p:discrete is enough finite} provides an analogous result for a discrete finite measure space.
\begin{proposition}\label{p:[0,1] is enough}
	Let $I$ be a set, $\tau\colon \R^I\to \R$, with $\lvert I\rvert\leq \lvert\omega\rvert$, and $p\in [1,+\infty)$.
	Then $\tau$ preserves $p$-integrability over every finite measure space if, and only if, $\tau$ preserves $p$-integrability over $([0,1],\mathcal{B}_{[0,1]},\Leb)$.
\end{proposition}
\begin{proof}
	Trivially, if $\tau$ preserves $p$-integrability, then $\tau$ preserves $p$-integrability over the measure space $([0,1],\mathcal{B}_{[0,1]},\Leb)$.
	For the converse, by Proposition~\ref{p:R is quite important} and Remark~\ref{r:polish is quite important}, if $\tau$ preserves $p$-integrability over $([0,1],\mathcal{B}_{[0,1]},\Leb)$ then $\tau$ preserves measurability.
	By Remark~\ref{r:[0,1] is partbl}, the measure space $([0,1],\mathcal{B}_{[0,1]},\Leb)$ is conditionally partitionable.
	An application of $(2)\Rightarrow(1)$ in Lemma~\ref{l:one measure integrability finite} concludes the proof.
\end{proof}
Similarly to the case of arbitrary measure, we next provide an analogue of Proposition~\ref{p:[0,1] is enough} for a discrete finite measure space.
\begin{lemma}\label{l:exists measure finite}
	There exists a probability measure $\mu$ on $(\omega,\mathcal{P}(\omega))$ such that the measure space $(\omega,\mathcal{P}(\omega),\mu)$ is conditionally partitionable.
\end{lemma}
\begin{proof}
	Let $X \df \{(n, m) \in \omega \times \omega \mid m \geq n \}$.
	We let $\nu$ be the unique measure on $(X,\mathcal{P}(X))$ such that, for every $(n,m) \in X$, we have $\nu(\{(n,m)\})=\frac{1}{2^{m}}$.
	Then,
	\begin{equation*}
		\begin{split}
			\sum_{(n,m)\in X}\nu(\{(n,m)\})	& = \sum_{n\in \omega} \sum_{m \in \omega, m \geq n} \nu(\{(n,m)\}) = \\
											& = \sum_{n\in \omega} \sum_{m \in \omega, m \geq n} \frac{1}{2^{m}} = \\
											& = \sum_{n\in \omega} \frac{2}{2^{n}} = \\
											& = 4. \\
		\end{split}
	\end{equation*}
	Hence, $\nu$ is a finite measure.
	
	We prove that $(X, \mathcal{P}(X), \nu)$ is conditionally partitionable.
	For $n \in \omega$, let $b_n \df \frac{1}{2^{n-1}}$.
	Let $(a_n)_{n \in \omega}$ be a sequence of elements of $\R^+$  satisfying $a_n \leq b_n$ for every $n \in \omega$.
	For every $n \in \N$, since $0 \leq a_n \leq \frac{1}{2^{n-1}}$, there exists a subset $K_n$ of $\{k \in \omega \mid k \geq n\}$ such that $a_n = \sum_{k \in K_n} \frac{1}{2^k}$.
	Set $A_n\df\{(n,m)\mid m\in K_n \}$.
	Note that $A_n \seq X$.
	Then $\mu(A_n)=\sum_{m\in K_n}\mu(\{(n,m)\})=\sum_{m\in K_n}\frac{1}{2^m}=a_n$.
	Moreover, for any pair of distinct $n, m \in \omega$, the sets $A_n$ and $A_m$ are disjoint.
	This proves that $(X, \mathcal{P}(X), \nu)$ is conditionally partitionable.
	
	Define the measure $\frac{\nu}{4}$ on $(X, \mathcal{P}(X))$ by setting $\frac{\nu}{4}(A)=\frac{\nu(A)}{4}$. 
	Using the fact that $(X, \mathcal{P}(X), \nu)$ is a conditionally partitionable measure space, it is not difficult to see that $(X, \mathcal{P}(X), \frac{\nu}{4})$ is a conditionally partitionable measure space, too.
	We have $\frac{\nu}{4}(X) = \frac{\nu(X)}{4}=\frac{4}{4}$; thus $\frac{\nu}{4}$ is a probability measure.
			
	The set $X$ is countably infinite, hence $(X,\mathcal{P}(X))$ and $(\omega,\mathcal{P}(\omega))$ are isomorphic measurable spaces, which concludes the proof.
\end{proof}
\begin{proposition}\label{p:discrete is enough finite}
	There exists a probability  measure $\mu$ on $(\omega,\mathcal{P}(\omega))$ such that, for every set $I$, every function $\tau\colon \R^I\to \R$ and every $p\in [1,+\infty)$, $\tau$ preserves $p$-integrability over every finite measure space if, and only if, $\tau$ preserves measurability and $\tau$ preserves $p$-integrability over $(\omega,\mathcal{P}(\omega),\mu)$.
\end{proposition}
\begin{proof}
	By Lemma~\ref{l:exists measure finite}, there exists a probability measure $\mu$ on $(\omega,\mathcal{P}(\omega))$ such that $(\omega,\mathcal{P}(\omega),\mu)$ is conditionally partitionable.
	The thesis follows from $(1)\Leftrightarrow (2)$ in Lemma~\ref{l:one measure integrability finite}.
\end{proof}
%

%%%%%%%%%%%%%%%%%%%%%%%%%%%%%%%%%%%%%%%%%%%%%%%%%%%%%%%%%%%%%%%%%%%%%%%%%%%%%%%%%%%%
%SECTION
\section{Generation}\label{s:generation}

The goal of this section is to prove Theorems~\ref{t:generation} and~\ref{t:generation finite}, which exhibit a generating set for the class of operations that preserve integrability over arbitrary and finite measure spaces, respectively.

As it is shown by Theorems~\ref{t:MAIN} and~\ref{t:MAIN finite}, the fact that an operation preserves $p$-integrability -- over arbitrary and finite measure spaces, respectively -- does not depend on the choice of $p$.
Hence, we say that the operation \emph{preserves integrability}.

Recall from the introduction the operation 
\[
	\truncsup(y,x_0,x_1,\dots)\df\sup_{n\in \omega }\{x_n\land y \}.
\]
We adopt the notation 
\[
	\truncsup_{n \in \omega  }^y x_n\df\truncsup(y,x_0,x_1,\dots).
\]
From the operations $0$, $+$, $\lor$ and $\lambda ({\;\cdot\;})$ (for each $\lambda \in \R$) we generate the operations 
\begin{equation*}
	\begin{split}
		f\land g		& \df -((-f)\lor (-g)), \\
		f^+				& \df f\lor 0, \\
		f^-				& \df - (f\land 0), \\
		\lvert f \rvert	& \df f^+-f^-.
	\end{split}
\end{equation*}
Additionally, using $\truncsup$, we generate 
\[
	\bigcurlywedge\limits_{n\in \omega }^g f_n\df\inf_{n\in \omega }\{f_n\lor g \}=-\truncsup_{n\in\omega}^{-g} -f_n.
\]
Let $\Omega$ be a set and let $S\seq \R^\Omega$.
We let $\sigma (S)$ denote the smallest $\sigma$-algebra $\mathcal{F}$ of subsets of $\Omega$ such that every $s\in S$ is $\mathcal{F}$-measurable.
\begin{lemma}\label{l:generation2}
	Let $\Omega$ be a set and $S\seq \R^\Omega$.
	Then $\sigma(S)$ is the $\sigma$-algebra of subsets of $\Omega$ generated by $\left\{g^{-1}((\lambda,+\infty))\mid  g\in S,\lambda \in \R\right\}$.
\end{lemma}
\begin{proof}
	See Proposition~2.3 in \cite{Folland}.
\end{proof}
\begin{lemma}\label{l:restriction}
	Let $\Omega$ be a set, let $\mathcal{A}\seq \mathcal{P}(\Omega)$, let $K$ be an element of the $\sigma$-algebra of subsets of $\Omega$ generated by $\mathcal{A}$, and let $K\seq Y\seq \Omega$.
	Then $K$ belongs to any $\sigma$-algebra $\mathcal{G}$ of subsets of $Y$ such that $A\cap Y\in \mathcal{G}$ for each $A\in \mathcal{A}$.
\end{lemma}
\begin{proof}
	Let $\Sigma\df \{S\seq \Omega\mid S\cap Y\in \mathcal{G} \}$.
	A straightforward verification shows that $\Sigma$ is a $\sigma$-algebra of subsets of $\Omega$.
	Moreover, $\mathcal{A}\seq \Sigma$.
	Therefore, by definition of $\mathcal{F}$, $\mathcal{F}\seq \Sigma$.
	Hence, $K\in \Sigma$, which means $K=K\cap Y\in \mathcal{G}$.
\end{proof}
Given $S\seq \R^\Omega$, we denote by $\langle S\rangle$ the closure of $S$ under $0$, $+$, $\lor$, $\lambda ({\;\cdot\;})$ (for each $\lambda \in \R$),  $\truncsup$ and $\trunc$.
Given $A\seq \Omega$, we write $\ind_A$ for the characteristic function of $A$ in $\Omega$.
\begin{lemma}\label{l:sigmaG}
	Let $\Omega$ be a set, let $S\seq \R^\Omega$, let $K\in \sigma(S)$ and let $K\seq Y\seq\Omega$ be such that $\ind_Y\in \langle S\rangle$.
	Then $\ind_{K}\in \langle S\rangle$.
\end{lemma}
\begin{proof}
	Set $\mathcal{G}\df \{C\seq Y\mid \ind_C\in \langle S\rangle \}$.
	Note that $\mathcal{G}$ is a $\sigma$-algebra of subsets of $Y$.
	Indeed, $\ind_Y\in \langle S\rangle$, and, for $C_0,C_1\seq Y$, we have $\ind_{C_0\cap C_1}=\ind_{C_0}\land\ind_{C_1}$ and $\ind_{Y\setminus C_0}=\ind_Y- \ind_{C_0}$.
	Further, let $(C_n)_{n\in \omega}$ be a family with $C_n\seq Y$.
	The characteristic function of $\bigcup\limits_{n\in\omega}C_n$ is $\truncsup_{n\in\omega}^{\ind_Y}\ind_{C_n}$.
	
	By Lemma~\ref{l:generation2}, the $\sigma$-algebra $\sigma(S)$ is generated by $\mathcal{A}\df\{g^{-1}((\lambda,+\infty))\mid g\in S,\lambda\in \R \}$.
	Let $A\in \mathcal{A}$, and write $A=g^{-1}((\lambda,+\infty))$ for some $g\in S$ and some $\lambda\in \R^+$.
	We have
	\begin{equation}\label{e:ind a cap y}
		\ind_{A\cap Y}\coloneqq \truncsup_{n\in\omega}^{\ind_Y} n(g- \lambda\ind_Y)^+.
	\end{equation}
	Indeed, for $x\in A\cap Y$, we have $g(x)>\lambda$ and $\ind_Y(x)=1$, hence 
	\[
		\truncsup_{n\in\omega}^{\ind_Y(x)} n(g(x)- \lambda \ind_Y(x) )^+=\truncsup_{n\in\omega}^{1} n(\underbrace{g(x)- \lambda}_{>0})^+=1.
	\]
	For $x\in \Omega\setminus Y$, we have $\ind_Y(x)=0$, and therefore 
	\[
		\truncsup_{n\in\omega}^{\ind_Y(x)} n(g(x)- \lambda \ind_Y (x) )^+=\truncsup_{n\in\omega}^{0} n(g(x))^+=0.
	\]
	For $x\in Y\setminus A$, we have $g(x)\leq \lambda$ and $\ind_Y(x)=1$, hence
	\[
		\truncsup_{n \in \omega  }^{\ind_Y(x)}n(g(x)-\lambda\ind_Y(x))^+=\truncsup_{n \in \omega  }^1n(g(x)-\lambda)^+=\truncsup_{n \in \omega }^1 0=0.
	\]
	Given equation \eqref{e:ind a cap y}, we have $\ind_{A\cap Y}\in \langle S\rangle$, which means $A\cap Y\in \mathcal{G}$.
	By Lemma~\ref{l:restriction}, $K\in\mathcal{G}$.
\end{proof}
The truncation operation $\overline{\;\cdot\;}$ comes into play in the following lemma.
\begin{lemma}\label{l: than lambda}
	Let $\lambda\in \R^+\setminus \{0\}$.
	The operations 
	\[
		\ind_{{\;\cdot\;}> \lambda}\colon\R\longrightarrow \R,\ \ x\longmapsto
																				\begin{cases}
																					1&\text{if }x> \lambda;\\
																					0&\text{otherwise,} 
																				\end{cases}
	\]
	and 
	\[
		\ind_{{\;\cdot\;} \geq \lambda} \colon \R \longrightarrow \R,\ \ x \longmapsto
																						\begin{cases}
																							1&\text{if }x\geq \lambda;\\
																							0&\text{otherwise,} 
																						\end{cases}
	\]
	are generated by the operations $0$, $+$, $\lor$, $\lambda ({\;\cdot\;})$ (for each $\lambda \in \R$), $\truncsup$, $\trunc$.
\end{lemma}
\begin{proof}
	Computation shows $\ind_{f >1}=\truncsup_{n\in\omega}^{\overline{f}}n(f-\overline{f})$.
	Moreover, $\ind_{f >\lambda}=\ind_{\frac{1}{\lambda} f>1}$.
	Finally, let $0< q_0< q_1<\cdots$ be a sequence of elements of $\R$ such that $q_n\to \lambda$.
	Then $\ind_{f \geq\lambda} = \bigcurlywedge\limits_{n\in \omega}^{0}\ind_{f >q_n}$.
\end{proof}
\begin{lemma}\label{l:simple}
	Let $S \seq \R^\Omega$, let $g\in \langle S\rangle$, $A\in \sigma(S)$, $\lambda\in \R^+$ be such that $\lambda\ind_A\leq g$.
	Then $\lambda\ind_A\in \langle S\rangle$.
\end{lemma}
\begin{proof}
	We have $0\in \langle S\rangle$, hence the thesis is immediate for $\lambda=0$.
	Suppose $\lambda>0$.
	Then $A\seq \{x \in \Omega\mid  g(x)\geq \lambda\}$.
	By Lemma~\ref{l: than lambda}, $\ind_{\{x \in \Omega\mid  g(x)\geq \lambda\}}=\ind_{g\geq \lambda}\in \langle S\rangle$.
	By Lemma~\ref{l:sigmaG}, $\ind_A\in \langle S\rangle$, hence $\lambda\ind_A\in \langle S\rangle$.
\end{proof}
\begin{lemma}\label{l:f under g}
	Let $S\seq\R^\Omega$, let $g\in \langle S\rangle$ and let $f\in \R^\Omega$ be $\sigma(S)$-measurable and such that $\lvert f\rvert\leq g$.
	Then $f\in \langle S\rangle$.
\end{lemma}
\begin{proof}
	First, we prove the statement for $f\geq0$.	Given that $f$ is  positive and $\sigma(S)$-measurable, $f$ is the supremum in $\R^\Omega$ of a positive increasing sequence $(s_n)_{n\in\omega}$ of $\sigma(S)$-measurable simple functions (see \cite{Rudin}, Theorem~1.17).
	By Lemma~\ref{l:simple}, $s_n\in \langle S\rangle$ for every $n\in \omega$.
	Hence
	\[
		f=\sup_{n\in \omega }s_n=\sup_{n\in \omega }s_n\land g=\truncsup_{n\in\omega}^g s_n\in \langle S\rangle.
	\]
	For $f$ not necessarily positive, the previous part of the proof shows that $f^+$ and $f^-$ belong to $\langle S\rangle$.
	Then $f=f^+-f^-\in \langle S\rangle$.
\end{proof}
\begin{lemma}\label{l:sup is mbl}
	Let $(\Omega,\mathcal{F})$ be a measurable space, and, for each $n\in \omega$, let $f_n\colon \Omega\to \R$ be a measurable function.
	If, for every $x\in \Omega$, $\sup_{n\in \omega} f_n(x)\in \R$, then $\sup f_n\colon \Omega\to \R$ is measurable.
	Analogously, if, for every $x\in \Omega$, $\inf_{n\in \omega} f_n(x)\in \R$, then the function $\inf_{n\in \omega} f_n\colon \Omega\to \R$ is measurable.
\end{lemma}
\begin{proof}
	By \cite{Rudin}, Theorem~1.14.
\end{proof}
\begin{lemma}\label{l:preserve integrability}
	The operations $0,+,\lor, \lambda(\;\cdot\;)$ (for each $\lambda \in \R$),  $\truncsup$ and $\trunc$ preserve integrability.
\end{lemma}
\begin{proof}
	The operations $0,+,\lor,\lambda(\;\cdot\;)$ (for each $\lambda \in \R$) and $\trunc$ preserve integrability.
	Moreover, $\truncsup^{g}\limits_{n\in \omega }f_n=\sup_{n\in \omega }\{f_n\land g\}$ and therefore, by Lemma~\ref{l:sup is mbl}, $\truncsup$ preserves measurability.
	The constant function $0$ is always integrable, therefore $0$ preserves integrability.
	By (3) in Lemma~\ref{l:implies lp}, $+$ preserves integrability.
	The operation $\lvert \;\cdot\;\rvert$ is immediately seen to preserve integrability.
	Since, for every $f,g$ functions, $\lvert f\lor g\rvert\leq \lvert f\rvert +\lvert g\rvert$,  then $\lor$ preserves integrability by (4) in Lemma~\ref{l:implies lp}.
	We have $\truncsup_{n\in \omega}^g f_n=\sup_{n\in \omega}\{f_n \land g\}$, and therefore $f_0\land g\leq  \truncsup_{n\in \omega}^g f_n \leq g$.
	Hence, $\lvert  \truncsup_{n\in \omega}^g f_n\rvert\leq\lvert g\rvert + \lvert f_0\rvert$.
	Thus, $\truncsup$ preserves integrability.
	Finally, $\lvert \overline{f}\rvert\leq \lvert f\rvert$, and therefore $\overline{\;\cdot\;}$ preserve integrability, by (4) in Lemma~\ref{l:implies lp}.
\end{proof}
\begin{proof}[Proof of Theorem~\ref{t:generation}]
	The operations $0$, $+$, $\lor$, $\lambda ({\;\cdot\;})$ (for each $\lambda \in \R$),  $\truncsup$ and $\trunc$ preserve integrability by Lemma~\ref{l:preserve integrability}.
	Moreover, by definition, the class of integrability-preserving operations is closed under every integrability-preserving operation and contains the projection functions.
	Therefore, every operation generated by $0$, $+$, $\lor$, $\lambda ({\;\cdot\;})$ (for each $\lambda \in \R$),  $\truncsup$ and $\trunc$ preserves integrability.
	
	To prove the converse, we use Theorem~\ref{t:MAIN}.
	Let $J$ be a finite subset of $I$, and let $(\lambda_j)_{j \in J}$ be a $J$-tuple of nonnegative real numbers.
	Then $\sum_{j \in J} \lambda_j \lvert \pi_{j}\rvert \in \langle\{\pi_i\mid i\in I \}\rangle$.
	Let $\tau$ be $\Cyl\left(\R^I\right)$-measurable and such that for every $v\in \R^I$ we have $\lvert\tau(v)\rvert\leq \sum_{j \in J} \lambda_j \lvert v_j\rvert$, i.e.,  $\lvert\tau\rvert\leq \sum_{j \in J} \lambda_j \lvert \pi_{j}\rvert$.
	Note that $\Cyl\left(\R^I\right)=\sigma(\{\pi_i\mid i\in I\})$, by definition.
	Then $\tau\in \langle\{\pi_i\mid i\in I \}\rangle$, by Lemma~\ref{l:f under g}.
	Therefore, $\tau$ is generated by $0$, $+$, $\lor$,  $\lambda ({\;\cdot\;})$ (for each $\lambda \in \R$),  $\truncsup$, $\trunc$.
\end{proof}
It is worth recalling that, in the proof of Theorem~\ref{t:generation}, the role of the truncation operation $\overline{\;\cdot\;}$ lies in Lemma~\ref{l: than lambda}.
\begin{proof}[Proof of Theorem~\ref{t:generation finite}]
	Note that the operations $0$, $+$, $\lor$, $\lambda ({\;\cdot\;})$ (for each $\lambda \in \R$),  $\truncsup$ and $1$ preserve integrability over finite measure spaces.
	Moreover, by definition, the class of the  operations that preserve integrability over finite measure spaces is closed under every integrability-preserving operation and contains the projection functions.
	Therefore, every operation generated by $0$, $+$, $\lor$, $\lambda ({\;\cdot\;})$ (for each $\lambda \in \R$),  $\truncsup$ and $1$ preserves integrability over every finite measure space.
	
	To prove the converse, we use Theorem~\ref{t:MAIN finite}.
	Note that the truncation is generated by $\lor$, $-1(\;\cdot\;)$ (i.e., scalar multiplication by $-1$), and $1$; indeed, $\overline{f}=f\land 1=-((-f)\lor(-1 ))$.
	Let $J$ be a finite subset of $I$, let $(\lambda_j)_{j \in J}$ be a $J$-tuple of nonnegative real numbers, and let $k \in \R^+$.
	Then $k + \sum_{j \in J} \lambda_j \lvert \pi_{j}\rvert \in \langle\{\pi_i\mid i\in I \} \cup \{ k \}\rangle$.
	Let $\tau$ be $\Cyl\left(\R^I\right)$-measurable and such that for every $v\in \R^I$ we have $\lvert\tau(v)\rvert\leq k + \sum_{j \in J} \lambda_j \lvert v_j\rvert$, i.e.,  $\lvert\tau\rvert\leq k + \sum_{j \in J} \lambda_j \lvert \pi_{j}\rvert$.
	Note that $\Cyl\left(\R^I\right)=\sigma(\{\pi_i\mid i\in I\})=\sigma(\{\pi_i\mid i\in I\}\cup\{1\})$, by definition.
	Then $\tau\in \langle\{\pi_i\mid i\in I \}\cup \{1\}\rangle$, by Lemma~\ref{l:f under g}.
	Therefore, $\tau$ is generated by $0$, $+$, $\lor$, $\lambda -$ (for each $\lambda \in \R$),  $\truncsup$, $1$.
\end{proof}
%

%&&&&&&&&&&&&&&&&&&&&&&&&&&&&&&&&&&&&&&&&&&&&&&&&&&&&&&&&&&&&&&&&&&&&&&&&&&&&&&&&&&&
%PART
\part{Truncated Riesz spaces and weak units}\label{part:variety}

%%%%%%%%%%%%%%%%%%%%%%%%%%%%%%%%%%%%%%%%%%%%%%%%%%%%%%%%%%%%%%%%%%%%%%%%%%%%%%%%%%%%
%SECTION
\section{The operation $\truncsup$}\label{s:truncsup}

We now investigate the operation $\truncsup$, defined on $\R$ in Section~\ref{s:generation}, for more general lattices.
Given a Dedekind $\sigma$-complete (not necessarily bounded) lattice $B$ we write $\truncsup$ for the operation on $B$ of countably infinite arity defined as
\[
	\truncsup(g,f_0,f_1,\dots)\df\sup_{n\in \omega}\{f_n\land g \}.
\]
We adopt the notation 
\[
	\truncsup_{n\in\omega}^g f_n\df \truncsup(g,f_0,f_1,\dots).
\]
\begin{proposition}\label{p:if ded then a}
	If $B$ is a Dedekind $\sigma$-complete lattice, the following properties hold for every $g,h\in B$, $(f_n)_{n\in \omega}\seq B$.
	\begin{enumerate}[label = {\rm (TS\arabic*)}, ref = TS\arabic*]
		\item \label{i:TS1} $\truncsup\limits_{n \in \omega  }^g f_n=\truncsup\limits_{n \in \omega  }^g (f_n\land g)$.
		\item \label{i:TS2} $\truncsup\limits_{n\in \omega }^g f_n=(f_0\land g)\lor\left(\truncsup\limits_{n\in \omega\setminus \{0\}}^g f_n\right)$.
		\item \label{i:TS3} $\truncsup\limits_{n \in \omega  }^g (f_n\land h)\leq h$.
	\end{enumerate}
\end{proposition}
\begin{proof}
	Straightforward verification.
\end{proof}
\noindent Conversely, we have the following.
\begin{proposition}\label{p:if a then Ded}
	If $B$ is a lattice endowed with an operation $\truncsup$ of countably  infinite arity 	which satisfies \eqref{i:TS1}, \eqref{i:TS2} and \eqref{i:TS3}, then $B$ is Dedekind $\sigma$-complete and $\truncsup_{n\in\omega}^g f_n=\sup_{n\in \omega}\{f_n\land g \}.$
\end{proposition}
\begin{proof}
	By induction on $k\in \omega$, \eqref{i:TS2} entails
	\[
		\truncsup_{n\in \omega}^g f_n=(f_0\land g)\lor\dots\lor (f_k\land g) \lor\left(\truncsup_{n\geq k+1 }^g f_{n}\right).
	\]
	Thus $f_k\land g\leq (f_0\land g)\lor\dots\lor (f_k\land g) \lor\left(\truncsup_{n\geq k+ 1 }^g f_{n}\right) = \truncsup_{n\in \omega}^g f_n$.
	Thus, $\truncsup_{n \in \omega }^g f_n$ is an upper bound of $(f_k\land g)_{k\in\omega  }$.
	Suppose now that $f_n\land g\leq h$ for every $n\in \omega  $.
	Then
	\[
		\truncsup_{n\in \omega  }^g f_n\stackrel{\text{\eqref{i:TS1}}}{=}
		\truncsup_{n\in \omega }^g (f_n\land g) \stackrel{ f_n\land g\leq h}{=}
		\truncsup_{n\in \omega  }^g (f_n\land g\land h)\stackrel{\text{\eqref{i:TS3}}}{\leq}
		h.
	\]
	This shows $\truncsup_{n\in \omega}^g f_n=\sup_{n \in \omega} \{f_n\land g \}$.	To prove that $B$ is Dedekind $\sigma$-complete, let $(f_n)_{n\in \omega }\subseteq B$ and $g\in B$ be such that $f_n\leq g$ for all $n\in \omega$.
	Then 
	\[
		\truncsup_{n \in \omega}^g f_n = 
		\sup_{n \in\omega}\{f_n\land g\} \stackrel{f_n\leq g}{=} 
		\sup_{n \in \omega} f_n.
	\]
\end{proof}
A map between two partially ordered sets is \emph{$\sigma$-continuous} if it preserves all existing countable suprema.
\begin{proposition}\label{p:sigma-cont}
	Let $\varphi\colon B\to C$ be a lattice morphism  between two Dedekind $\sigma$-complete lattices.
	Then $\varphi$ is $\sigma$-continuous if, and only if, $\varphi$ preserves $\truncsup$.
\end{proposition}
\begin{proof}
	First, suppose $\varphi$ preserves $\truncsup$.
	Let $(f_n)_{n\in \omega }\subseteq B$ and $f=\sup_{n \in \omega} f_n$.
	Then
	\begin{equation*}
		\begin{aligned}
			\varphi \left(\sup_{n\in \omega}f_n \right)	& = \varphi \left(\sup_{n\in \omega}\{f_n\land f\}\right) =	& \ \ \ \ \ \ (\text{because } f_n\leq f) \\
														& = \varphi \left(\truncsup_{n \in \omega}^f f_n\right) = 	& \ \ \ \ \ \ \\
														& = \truncsup_{n \in \omega}^{\varphi(f)} \varphi(f_n) =	& \ \ \ \ \ \ (\text{because } \varphi \text{ preserves }\truncsup) \\
														& = \sup_{n \in \omega} \{\varphi(f_n) \land \varphi(f)\} =	& \ \ \ \ \ \ \\
														& = \sup_{n \in \omega} \varphi(f_n \land f) =				& \ \ \ \ \ \ (\text{because } \varphi \text{ preserves }\land) \\
														& = \sup_{n \in \omega} \varphi(f_n).						& \ \ \ \ \ \ (\text{because } f_n\leq f)
		\end{aligned}
	\end{equation*}
	Therefore, $\varphi$ is $\sigma$-continuous.
	
	For the converse implication, suppose that $\varphi$ is $\sigma$-continuous.
	Let $(f_n)_{n\in \omega  }\subseteq B$ and $g\in B$.	
	Then
	\begin{equation*}
		\begin{aligned}
			\varphi \left(\truncsup_{n \in \omega}^g f_n\right)	& = \varphi\left(\sup_{n\in \omega} \{f_n\land g\}\right) =	& \\
																& = \sup_{n\in \omega} \varphi(f_n\land g) =				& \text{(because $\varphi$ preserves count.\ sups)} \\
																& = \sup_{n\in \omega} \{\varphi(f_n)\land \varphi(g)\} =	& \text{(because $\varphi$ preserves $\land$)} \\
																& = \truncsup_{n\in \omega}^{\varphi(g)} \varphi(f_n).		&
		\end{aligned}
	\end{equation*}
	Hence, $\varphi$ preserves $\truncsup$.
\end{proof}
\begin{remark}\label{r:Dedekind is variety}
	Propositions~\ref{p:if ded then a},~\ref{p:if a then Ded} and~\ref{p:sigma-cont} show that, whenever  $\mathcal{V}$ is a variety with a lattice reduct, then its subcategory of Dedekind $\sigma$-complete objects, with $\sigma$-continuous morphisms, is a variety which has, as primitive operations, the operations of $\mathcal{V}$ together with $\truncsup$, and, as axioms, the axioms of $\mathcal{V}$ together with \eqref{i:TS1}, \eqref{i:TS2} and \eqref{i:TS3}.
\end{remark}
%

%%%%%%%%%%%%%%%%%%%%%%%%%%%%%%%%%%%%%%%%%%%%%%%%%%%%%%%%%%%%%%%%%%%%%%%%%%%%%%%%%%%%
%SECTION
\section{Truncated $\ell$-groups}\label{s:defn truncated}

We assume familiarity with the basic theory of $\ell$-groups.
All needed background can be found, for example, in the standard reference \cite{BKW}.
In \cite{Ball1}, R.\ N.\ Ball defines a truncated $\ell$-group as an abelian divisible $\ell$-group that is endowed with a function $\overline{\;\cdot\;}\colon G^+\to G^+$, called \emph{truncation}, which has the following properties for all $f,g\in G^+$.
\begin{enumerate}[label = (B\arabic*), ref = B\arabic*]
	\item \label{i:B1} $f\land \overline{g}\leq \overline{f}\leq f$.
	\item \label{i:B2} If $\overline{f}=0$, then $f=0$.
	\item \label{i:B3} If $nf =\overline{nf}$ for every $n\in \omega$, then $f=0$.
\end{enumerate}
In this paper, we do not assume divisibility.
The truncation $\overline{\;\cdot\;}$ may be extended to an operation on $G$, by setting $\overline{f}=\overline{f^+}-f^-$.
Here, as is standard, we set $f^+\df f\lor 0$, and $f^-\df-(f\land 0)$.
Then, Ball's definition may be reformulated as follows.

\begin{definition}
	A \emph{truncated $\ell$-group} is an abelian $\ell$-group that is endowed with a unary operation $\overline{\;\cdot\;}\colon G\to G$, called \emph{truncation}, which has the following properties.
	\begin{enumerate} [label = {\rm (T\arabic*)}, ref = T\arabic*]
		\item \label{i:T1} For all $f\in G$, we have $\overline{f}=\overline{f^+}-f^-$.
		\item \label{i:T2} For all $f\in G^+$, we have $\overline{f}\in G^+$.
		\item \label{i:T3} For all $f, g\in G^+$, we have $f\land \overline{g}\leq \overline{f}\leq f$.
		\item \label{i:T4} For all $f\in G^+$, if $\overline{f}=0$, then $f=0$.
		\item \label{i:T5} For all $f\in G^+$, if $nf =\overline{nf}$ for every $n\in \omega$, then $f=0$.	\qed
	\end{enumerate}
\end{definition}
Axiom \eqref{i:T2} ensures that $\overline{\;\cdot\;}$ may be restricted to an operation on $G^+$.
Axiom \eqref{i:T1} gives the one-to-one correspondence with Ball's definition.
Axioms \eqref{i:T3}, \eqref{i:T4}, \eqref{i:T5} correspond, respectively, to Axioms \eqref{i:B1}, \eqref{i:B2}, \eqref{i:B3}.
An $\ell$-homomorphism $\varphi$ between truncated $\ell$-groups preserves $\overline{\;\cdot \;}$ if, and only if, $\varphi$ preserves $\overline{\;\cdot\;}$ over positive elements; indeed, if $\varphi$ preserves $\overline{\;\cdot\;}$ over positive elements, then, for $f\in G$, $\varphi(\overline{f})=\varphi(\overline{f^+}-f^-)=\varphi(\overline{f^+})-\varphi(f^-)=\overline{\varphi(f^+)}-\varphi(f^-)=\overline{\varphi(f)^+}-\varphi(f)^-=\overline{\varphi(f)}$.
This ensures that the equivalence with Ball's definition also holds for morphisms.

Note that \eqref{i:T1}, \eqref{i:T2} and \eqref{i:T3} are (essentially) equational axioms. This is evident for \eqref{i:T1}; \eqref{i:T2} can be written as $\forall f\  \overline{f^+} \land 0 = 0$; \eqref{i:T3} is the conjunction of the two equations $\forall f, g\  f^+\land \overline{g^+} \lor \overline{f^+} = \overline{f^+}$ and $\forall f\ \overline{f^+} \lor f^+= f^+$.
The axioms \eqref{i:T4} and \eqref{i:T5} cannot be expressed in such equational terms.
However, as we shall see, this becomes possible when we add the hypothesis of Dedekind $\sigma$-completeness.

It is well-known that a Dedekind $\sigma$-complete $\ell$-group is archimedean and thus abelian.
Let $G$ be a Dedekind $\sigma$-complete $\ell$-group, endowed with a unary operation $\trunc$.
We denote by (T4$'$) and (T5$'$) the following properties, which may or may not hold in $G$.
\begin{enumerate}[label = {\rm (T\arabic*$'$)}, start = 4, ref = T\arabic*$'$]
	\item \label{i:T4prime} For all $f\in G^+$, we have $f=\truncsup_{n\in \omega}^fn\overline{f}$.
	\item \label{i:T5prime} For all $f\in G^+$, we have $f =\truncsup_{n\in \omega}^{f}(nf-\overline{nf})$.
\end{enumerate}
Note that \eqref{i:T4prime} and \eqref{i:T5prime}, are (essentially) equational axioms: indeed, \eqref{i:T4prime} is equivalent to $\forall f \ \, f^+=\truncsup_{n\in \omega}^{f^+}n\overline{f^+}$, and \eqref{i:T5prime} is equivalent to $\forall f\ f^+ =\truncsup_{n\in \omega}^{f^+}(nf^+-\overline{nf^+})$.

Our aim in this section, met in Propositions~\ref{p:T4 and T5 then '},~\ref{p:B4} and~\ref{p:B5}, is to show that, for a Dedekind $\sigma$-complete $\ell$-group endowed with a unary operation $\trunc$ which satisfies \eqref{i:T1}, \eqref{i:T2} and \eqref{i:T3}, the Axioms \eqref{i:T4} and \eqref{i:T5} may be equivalently replaced by the equational axioms \eqref{i:T4prime} and \eqref{i:T5prime}.
This will show the axioms of Dedekind $\sigma$-complete truncated $\ell$-groups to be equational.
\begin{proposition}\label{p:T4 and T5 then '}
	Let  $G$ be an abelian $\ell$-group endowed with a unary operation $\trunc$.
	Then \eqref{i:T4prime} implies \eqref{i:T4}, and \eqref{i:T5prime} implies \eqref{i:T5}.
\end{proposition}
\begin{proof} 
	Suppose \eqref{i:T4prime}.
	Let $f\in G^+$ be such that $\overline{f}=0$.
	By \eqref{i:T4prime}, $f=\truncsup_{n\in\omega}^f n\overline{f}=\truncsup_{n\in\omega}^f 0=0$.
	Hence, \eqref{i:T4} holds.
	Suppose \eqref{i:T5prime}.
	Let $f\in G^+$ be such that $nf=\overline{nf}$ for every $n\in \omega$.
	By \eqref{i:T5prime}, $f=\truncsup_{n\in \omega}^{f}\left(nf-\overline{nf}\right)=\truncsup_{n\in \omega}^{f}0=0$.
	Hence \eqref{i:T5} holds.
\end{proof}
We shall use the following standard distributivity result.
\begin{lemma}\label{l:distriutivity a la BKW}
	Let $G$ be an $\ell$-group, $I$ a set and $(x_i)_{i\in I}\subseteq G$.
	If $\sup_{i\in I}x_i$ exists, then, for every $a\in G$, $\sup_{i\in I} \{a\land x_i \}$ exists and
	\[
		a\land \left(\sup_{i\in I} x_i\right)=\sup_{i\in I}\{a\land x_i \}.
	\]
\end{lemma}
\begin{proof}
	See Proposition~6.1.2 in \cite{BKW}.
\end{proof}
\begin{lemma}\label{l:sum distrib}
	Let $G$ be a Dedekind $\sigma$-complete $\ell$-group, let $g\in G$, $h\in G^+$ and $(f_n)_{n\in \omega}\seq G $.
	Then	
	\[
		\truncsup_{n\in \omega}^g(f_n+h)=\left(\left(\truncsup^g_{n\in \omega}f_n\right)+h\right)\land g.
	\]
\end{lemma}
\begin{proof}
	\begin{equation*}
		\begin{aligned}
			\truncsup_{n\in \omega}^g(f_n+h)	& = \sup_{n \in \omega}\{(f_n+h)\land g \} =							& \ \ \ \ \ \ \\
												& = \sup_{n \in \omega}\{(f_n+h)\land (g+h) \land g\} =					& \ \ \ \ \ \ (\text{because } h \geq 0) \\
												& = \sup_{n \in \omega}\{(f_n+h)\land (g+h) \}\land g =					& \ \ \ \ \ \ \text{(by Lemma~\ref{l:distriutivity a la BKW})} \\
												& = \sup_{n \in \omega}\{(f_n\land g)+h \}\land g = 					& \ \ \ \ \ \ \\
												& = \left(\sup_{n \in \omega}\{f_n\land g \}+h\right)\land g=			& \ \ \ \ \ \ \\
												& = \left(\left(\truncsup^g_{n\in \omega}f_n\right)+h\right)\land g.	& \ \ \ \ \ \ 
		\end{aligned}
	\end{equation*}
\end{proof}
\begin{proposition}\label{p:B4}
	Let $G$ be a Dedekind $\sigma$-complete $\ell$-group endowed with a unary operation $\trunc$ such that \eqref{i:T2}, \eqref{i:T3} and \eqref{i:T4} hold.
	Then \eqref{i:T4prime} holds, i.e., for all $f\in G^+$, 
	\[
		f = \truncsup_{n\in \omega}^{f} n \overline{f}.
	\]
\end{proposition}
\begin{proof}
	By \eqref{i:T2}, $\overline{f}\in G^+$.
	Therefore $0\overline{f}\leq 1\overline{f}\leq 2 \overline{f}\leq 3\overline{f}\leq \dots$.
	Hence,
	\begin{equation*}
		\begin{aligned}
			\truncsup_{n \in \omega}^f n \overline{f}	& = \truncsup_{n \in \omega\setminus \{0\}}^f n\overline{f} = 									& \ \ \ \ \ \ \\
														& = \truncsup_{n \in \omega}^f (n + 1) \overline{f} = 											& \ \ \ \ \ \ \\
														& = \truncsup_{n \in \omega}^f \left(n \overline{f} + \overline{f}\right) =						& \ \ \ \ \ \ \\
														& = \left(\left(\truncsup_{n\in \omega}^f n\overline{f} \right) + \overline{f} \right) \land f.	& \ \ \ \ \ \ \text{(by Lemma~\ref{l:sum distrib})}
		\end{aligned}
	\end{equation*}
	Therefore, setting $b\df\truncsup_{n\in \omega}^f n\overline{f}$, we have
	\[
		0=\left(\left(b+\overline{f}\right)\land f\right)-b=\overline{f}\land \left(f- b\right)=\overline{f-b},
	\]
	where the last equality holds because, by \eqref{i:T3}, we have $\overline{f}\land \left(f- b\right)\leq\overline{f-b}$ and, for the opposite inequality, we have $\overline{f-b}\leq f-b$ and $\overline{f-b}=\overline{f-b}\land f\leq \overline{f}$.
	
	By \eqref{i:T4}, since $\overline{f-b}=0$, we have $f- b=0$, i.e., $f=\truncsup_{n\in \omega}^{f} n\overline{f}$.
\end{proof}
\begin{lemma}\label{l:inequality}
	Let $G$ be a Dedekind $\sigma$-complete $\ell$-group endowed with a unary operation $\trunc$ such that \eqref{i:T2} and \eqref{i:T3} holds.
	Let $a,b\in G^+$.
	Then
	\[
		\overline{a+b}\leq \overline{a}+\overline{b}.
	\]
\end{lemma}
\begin{proof}
	By \eqref{i:T3}, $\overline{a+b}\leq a+b$.
	By \eqref{i:T2}, $\overline{a+b}\geq0$, thus $b\land \overline{(a+b)}\geq 0$, and therefore $\overline{a+b}\leq \overline{a+b}+\left(b\land \overline{(a+b)}\right)$.
	Hence,
	\begin{equation*}
		\begin{aligned}
			\overline{a+b}
				& \leq	\left[(a + b) \land \left(a + \overline{(a + b)}\right)\right] \land \left[\overline{(a + b)} + \left(b \land \overline{(a + b)}\right)\right] =	& \ \ \ \ \ \ \\
				& =		\left[a + \left(b \land \overline{(a + b)}\right)\right] \land \left[\overline{(a + b)} + \left(b \land\overline{(a + b)}\right)\right] =	 		& \ \ \ \ \ \ \\
				& =		\left(a \land \overline{(a + b)}\right) + \left(b \land \overline{(a + b)}\right) \leq																& \ \ \ \ \ \ \\
				& \leq	\overline{a} + \overline{b}.																														& \ \ \ \ \ \ \text{(by \eqref{i:T3})}
		\end{aligned}
	\end{equation*}
\end{proof}
\begin{lemma}\label{l:pseudobasic properties 2}
	Let $G$ be an abelian $\ell$-group endowed with a unary operation $\trunc$ such that \eqref{i:T3} holds.
	Then, for all $a,b\in G^+$, if \(a\leq b\), then \(a-\overline{a}\leq b-\overline{b}\).
\end{lemma}
\begin{proof}
	Since $a\leq b$, we have $\overline{b}-b\leq\overline{b}-a$.
	By \eqref{i:T3}, $\overline{b}-b\leq 0$.
	Hence, 
	\begin{equation*}
		\begin{aligned}
			\overline{b} - b	& \leq	\left(\overline{b} - a\right) \land 0	=	& \ \ \ \ \ \ \\
								& =		\left(\overline{b} \land a\right) - a \leq	& \ \ \ \ \ \ \text{(because $+$ distributes over $\land$)}\\
								& \leq	\overline{a} - a.							& \ \ \ \ \ \ \text{(by \eqref{i:T3})}
		\end{aligned}
	\end{equation*}
\end{proof}
\begin{proposition}\label{p:B5}
	Let $G$ be a Dedekind $\sigma$-complete $\ell$-group endowed with a unary operation $\trunc$ such that \eqref{i:T2}, \eqref{i:T3} and \eqref{i:T5} hold.
	Then \eqref{i:T5prime} holds, i.e.,  for all $f\in G^+$,
	\[
		f = \truncsup_{n \in \omega}^f \left(nf - \overline{nf}\right).
	\]
\end{proposition}
\begin{proof}
	Let $k\in \omega$.
	By \eqref{i:T3} we have $0\leq kf-\overline{kf}$.
	We have 
	\begin{equation*}
		\begin{aligned}
			\truncsup_{n \in \omega}^f \left(nf - \overline{nf}\right)	
				& \geq	\truncsup_{n \in \omega \setminus \{0, \dots, k - 1\}}^f \left(nf - \overline{nf}\right) =				& \ \ \ \ \ \ \\
				& = \truncsup_{n\in\omega}^f \left((n + k) f - \overline{(n + k) f}\right) \geq									& \ \ \ \ \ \ \\
			 	& \geq	\truncsup_{n\in\omega}^f \left(nf - \overline{nf} + kf - \overline{kf}\right) \geq						& \ \ \ \ \ \ \text{(by Lemma~\ref{l:inequality})} \\
				& =		\left(\left(\truncsup_{n\in\omega}^f (nf - \overline{nf})\right) + kf - \overline{kf}\right) \land f.	& \ \ \ \ \ \ \text{(by Lemma~\ref{l:sum distrib})}	
		\end{aligned}
	\end{equation*}
	The opposite inequality is immediate.
	Therefore, setting $b\df \truncsup_{n\in\omega}^f\left(nf-\overline{nf}\right)$, we have $b=\left(b+kf-\overline{kf}\right)\land f$, which implies $0=\left(\left(b+kf-\overline{kf}\right)\land f\right)-b=\left(kf-\overline{kf}\right)\land (f-b)$.
	We set $a\df f-b$.
	We have $0\leq a\leq f$, because $0\leq b\leq f$.
	By \eqref{i:T3} and Lemma~\ref{l:pseudobasic properties 2}, $0\leq ka-\overline{ka}\leq kf-\overline{kf}$.
	Therefore, $0=\left(ka-\overline{ka}\right)\land a$.
	It is elementary that, in any abelian group, $x\land y=0$ implies $(nx)\land y=0$ for each $n\in \omega$.
	Therefore, $0=\left(ka-\overline{ka}\right)\land ka\stackrel{\text{\eqref{i:T2}}}{=}\left(ka-\overline{ka}\right)$.
	Hence, $ka=\overline{ka}$.
	Since $k$ is arbitrary, by \eqref{i:T5} we infer $a=0$, i.e.\ $f -\truncsup_{n\in\omega}^f\left(nf-\overline{nf}\right)=0$.
\end{proof}
To sum up, Propositions~\ref{p:T4 and T5 then '},~\ref{p:B4} and~\ref{p:B5} show that, for Dedekind $\sigma$-complete $\ell$-groups endowed with a unary operation $\trunc$, Axioms (\ref{i:T1}-\ref{i:T5}) are equivalent to Axioms (\ref{i:T1}-\ref{i:T3}) together with Axioms \eqref{i:T4prime} and \eqref{i:T5prime}.

We denote by $\Gt$ the category whose objects are Dedekind $\sigma$-complete truncated $\ell$-groups, and whose morphisms are $\sigma$-continuous $\ell$-homomorphisms that preserve $\overline{\;\cdot \;}$.
Since Axioms \eqref{i:T1}, \eqref{i:T2}, \eqref{i:T3}, \eqref{i:T4prime} and \eqref{i:T5prime} are equational, $\Gt$ is a variety, whose operations are the operations of $\ell$-groups, together with $\overline{\;\cdot\;}$ and $\truncsup$, and whose axioms are the axioms of $\ell$-groups, together with the following ones.
\begin{enumerate}[label = {\rm (TS\arabic*)}, ref = TS\arabic*]
	\item $\truncsup\limits_{n \in \omega  }^g f_n=\truncsup\limits_{n \in \omega  }^g (f_n\land g)$.
	\item  $\truncsup\limits_{n\in \omega }^g f_n=(f_0\land g)\lor\left(\truncsup\limits_{n\in \omega\setminus \{0\}}^g f_n\right)$.
	\item $\truncsup\limits_{n \in \omega  }^g (f_n\land h)\leq h$.
\end{enumerate}
\begin{enumerate}[label = {\rm (T\arabic*)}, ref = T\arabic*]
	\item For all $f\in G$, we have $\overline{f}=\overline{f^+}-f^-$.
	\item For all $f\in G^+$, we have $\overline{f}\in G^+$.
	\item For all $f, g\in G^+$, we have $f\land \overline{g}\leq \overline{f}\leq f$.
\end{enumerate}
\begin{enumerate}[label = {\rm (T\arabic*$'$)}, start = 4, ref = T\arabic*$'$]
	\item For all $f\in G^+$, we have $f=\truncsup\limits_{n\in \omega}^fn\overline{f}$.
	\item For all $f\in G^+$, we have $f =\truncsup\limits_{n\in \omega}^{f}\left(nf-\overline{nf}\right)$.
\end{enumerate}
%

%%%%%%%%%%%%%%%%%%%%%%%%%%%%%%%%%%%%%%%%%%%%%%%%%%%%%%%%%%%%%%%%%%%%%%%%%%%%%%%%%%%%
%SECTION
\section{The Loomis-Sikorski Theorem for truncated $\ell$-groups}\label{s:LS}

\begin{definition}
	Given a set $X$, a \emph{$\sigma$-ideal of subsets of $X$} is a set $\mathcal{I}$ of subsets of $X$ such that the following conditions hold.
	\begin{enumerate}
		\item $\emptyset\in\mathcal{I}$.
		\item $B\in\mathcal{I}, A\subseteq B\Rightarrow A\in\mathcal{I}$.
		\item $(A_n)_{n\in \omega  }\subseteq \mathcal{I}\Rightarrow \bigcup_{n\in \omega  }A_n\in\mathcal{I}$.\qed
	\end{enumerate}
\end{definition}
If $\mathcal{I}$ is a $\sigma$-ideal of subsets of $X$, we say that a property $P$ \emph{holds for $\mathcal{I}$-almost every $x\in X$} if $\{x\in X\mid P$ does not hold for $x \}\in \mathcal{I}$.
A $\sigma$-ideal $\mathcal{I}$ of subsets of $X$ induces on $\R^X$ an equivalence relation $\sim$, defined by $f\sim g$ if, and only if, $f(x)=g(x)$ for $\mathcal{I}$-almost every $x\in X$.
We write $\RXI$ for the quotient $\frac{\R^X}{\sim}$.
Every operation $\tau$ of countable arity on $\R$ induces an operation $\tilde{\tau}$ on $\RXI$, by setting $\tilde{\tau}\left(\left([f_i]_{\mathcal{I}}\right)_{i\in I}\right)\df [g]_{\mathcal{I}}$, where $g(x)=\tau\left(\left(f_i(x)\right)_{i\in I}\right)$.
The assumption that $\mathcal{I}$ is closed under countable unions guarantees that this definition is well posed.
Therefore, by Remark~\ref{r:Dedekind is variety}, $\RXI$ is a Dedekind $\sigma$-complete truncated $\ell$-group.

The aim of this section is to prove the following theorem.
\begin{theorem}[Loomis-Sikorski Theorem for truncated $\ell$-groups]\label{t:LS gr}
	Let $G$ be a Dedekind $\sigma$-complete truncated $\ell$-group.
	Then there exist a set $X$, a  $\sigma$-ideal $\mathcal{I}$ of subsets of $X$ and an injective $\sigma$-continuous $\ell$-homomorphism $\iota\colon G\hookrightarrow \RXI$ such that, for every $f\in G$, $\iota\left(\overline{f}\right)=\iota(f)\land [1]_\mathcal{I}$.
\end{theorem}
We will give a proof that is rather self-contained, with the main exception of the use of Theorem~\ref{t:Nepalese} below.
Anyway, we believe that a shorter (but not self-contained) way to prove Theorem~\ref{t:LS gr} above (even in the less restrictive hypothesis that $G$ is an archimedean truncated $\ell$-group) may be the following.
First, show that the divisible hull $G^d$ of $G$ admits a truncation that extends the truncation of $G$.
Then, embed $G^d$ in $\frac{\R^X}{\mathcal{I}}$ via Theorem~5.3.6.(1) in \cite{Ball1}.
Finally, using arguments similar to those in Theorem~6.2 in \cite{Mundici}, show that this embedding preserves all countable suprema.
\begin{theorem}[Loomis-Sikorski Theorem for Riesz spaces]\label{t:Nepalese}
	Let $G$ be a Dedekind $\sigma$-complete Riesz space.
	Then there exist a set $X$, a $\sigma$-ideal $\mathcal{I}$ of subsets of $X$ and an injective $\sigma$-continuous Riesz morphism $\iota\colon G\hookrightarrow \RXI$.
\end{theorem}
For a proof of Theorem~\ref{t:Nepalese} see \cite{Nepalese}, or \cite{LoomisRevisited} and \cite{SmallRieszSpaces}.
\begin{corollary}[Loomis-Sikorski Theorem for $\ell$-groups]\label{c:Nepalese groups}
	Let $G$ be a Dedekind $\sigma$-complete $\ell$-group.
	Then there exist a set $X$, a  $\sigma$-ideal $\mathcal{I}$ of subsets of $X$ and an injective $\sigma$-continuous $\ell$-homomorphism $\iota\colon G\hookrightarrow \RXI$.
\end{corollary}
\begin{proof}
	There exist a Dedekind $\sigma$-complete Riesz space $H$ and an injective $\ell$-morphism $\varphi\colon G\hookrightarrow H$ that preserves every existing supremum; see, e.g., \cite{EmbeddingGroupInRiesz}.
	Applying Theorem~\ref{t:Nepalese} to the Dedekind $\sigma$-complete Riesz space $H$, we obtain an injective $\sigma$-continuous Riesz morphism $\varphi'\colon H\hookrightarrow \RXI$.
	The composition $\iota=\varphi'\circ\varphi\colon G\hookrightarrow \RXI$ is an injective $\sigma$-continuous $\ell$-morphism, since both $\varphi$ and $\varphi'$ are injective $\sigma$-continuous $\ell$-morphisms.
\end{proof}
Our strategy to prove Theorem~\ref{t:LS gr} is the following.
Lemma~\ref{l:L-S countably} will prove Theorem~\ref{t:LS gr} for countably generated algebras.
This will imply that $\R$ generates the variety of Dedekind $\sigma$-complete truncated $\ell$-groups, and from this fact Theorem~\ref{t:LS gr} is derived.
\begin{lemma}\label{l:bounded}
	Let $G$ be a Dedekind $\sigma$-complete truncated $\ell$-group generated by a subset $S\seq G$.
	Then, for every $g\in G$, there exist $s_0,\dots,s_{n-1}\in S$ such that 
	\[
		\lvert g \rvert \leq \lvert s_0 \rvert + \cdots + \lvert s_{n-1} \rvert.
	\]
\end{lemma}
\begin{proof}
	Let $T\df \{h\in G\mid \exists s_0,\dots,s_{n-1}\in G:\lvert h\rvert\leq \lvert s_0\rvert + \cdots + \lvert s_{n-1}\rvert\}$.
	It is immediately seen that $S\seq T$.
	It is standard that $T$ is a convex $\ell$-subgroup of $G$.
	Moreover, for every $g\in G$, and every $(f_n)_{n\in \omega}\seq G$, we have the following.
	\begin{enumerate}
		\item \label{i:first} $\truncsup_{n\in \omega}^g f_n=\sup_{n\in \omega}\{f_n \land g\}$, and therefore $f_0\land g\leq  \truncsup_{n\in \omega}^g f_n \leq g$.
		Hence, 
		\begin{equation*}
			\begin{split}
				\lvert  \truncsup_{n\in \omega}^g f_n\rvert	& =		\left(\truncsup_{n\in \omega}^g f_n\right)\lor \left(-\truncsup_{n\in \omega}^g f_n\right) \leq\\
															& \leq	g\lor [- (f_0\land g)]\leq\\
															& \leq 	g\lor [(-f_0)\lor (-g)] \leq \\
															& \leq \lvert g\rvert \lor \lvert f_0\rvert.
			\end{split}
		\end{equation*}
		\item \label{i:second} $\lvert \overline{g}\rvert =\lvert \overline{g^+}-g^-\rvert\leq \lvert \overline{g^+}\rvert+\lvert g^-\rvert\stackrel{\text{\eqref{i:T2}}}{=}\overline{g^+}+g^-\stackrel{\text{\eqref{i:T3}}}{\leq} g^++ g^-=\lvert g\rvert.$
	\end{enumerate}
	Since $T$ is a convex $\ell$-subgroup of $G$, \eqref{i:first} and \eqref{i:second} imply that $T$ is closed under $\truncsup$ and $\trunc$.
\end{proof}
\begin{lemma}\label{l:admits sup}
	Let $X$ be a set, and $\mathcal{I}$ a $\sigma$-ideal of subsets of $X$.
	Let $({g}_n)_{n\in \omega}$ be a sequence of functions from $X$ to $\R$.
	Suppose that, for $\mathcal{I}$-almost every $x\in X$, $\sup_{n\in \omega  } {g}_n(x)\in \R$.
	Then the set $\{[{g}_n]_\mathcal{I}\mid n\in \omega\}$ admits a supremum in $\RXI$.
\end{lemma}
\begin{proof}
	Let $A\in \mathcal{I}$ be such that, for every $x\in X\setminus A$,  $\sup_{n\in \omega  } {g}_n(x)\in \R$.
	Let $v\colon X\to \R$ be any function such that, for every  $x\in X\setminus A$, $v(x)=\sup_{n\in \omega  } {g}_n(x)$.
	Then $[v]_{\mathcal{I}}$ is the supremum of $\{[{g}_n]_\mathcal{I}\mid n\in \omega\}$ in $\RXI$.
\end{proof}
\begin{lemma}\label{l:doppio bigvee}
	Let $G$ be a Dedekind $\sigma$-complete truncated $\ell$-group, let $f\in G^+$ and let $(f_i)_{i\in \omega}\seq G^+$.
	Then
	\[
		f = \truncsup_{i\in\omega}^f \left(if - \truncsup_{k \in \omega}^{if} \overline{f_k}\right).
	\]
\end{lemma}
\begin{proof}
	Trivially, $f\leq\truncsup_{i\in\omega}^f\left(if-\truncsup_{k\in\omega}^{if}\overline{f_k}\right)$.
	We prove the opposite inequality.
	By \eqref{i:T3}, for every $k\in \omega$, we have $\overline{f_k}\land (if)\leq \overline{if}$, and therefore we have $\truncsup_{k\in\omega}^{if}\overline{f_k}=\sup_{i\in \omega }\left\{\overline{f_k}\land (if)\right\}\leq \overline{if}.$
	Hence, $if-\truncsup_{k\in\omega}^{if}\overline{f_k}\geq if-\overline{if}$.
	Therefore, we have $\truncsup_{i\in\omega}^f\left(if-\truncsup_{k\in\omega}^{if}\overline{f_k}\right)\geq \truncsup_{i\in\omega}^f\left(if-\overline{if}\right)\stackrel{\text{\eqref{i:T5prime}}}{=}f$.
\end{proof}
\begin{lemma}\label{l:T1}
	Let $G$ be an abelian $\ell$-group, let $a\in G$ and let $u\in G^+$.
	Then, $(a^+\land u)-a^-=a\land u$.
\end{lemma}
\begin{proof}
	$(a^+\land u)-a^-=(a^+-a^-)\land (u-a^-)=a\land (u+(a\land 0))=a\land (u+a)\land u=a\land u.$
\end{proof}
\begin{lemma}\label{l:the big step}
	Let $G$ be a countably generated Dedekind $\sigma$-complete truncated $\ell$-group.
	Then there exist a set $X$, a  $\sigma$-ideal $\mathcal{I}$ of subsets of $X$, an injective $\sigma$-continuous $\ell$-homomorphism $\iota\colon G\hookrightarrow \RXI$ and an element $u\in \RXI$ such that, for every $f\in G$, 
	\[
		\iota\left(\overline{f}\right)=\iota(f)\land u.
	\]
\end{lemma}
\begin{proof}
	By Corollary~\ref{c:Nepalese groups}, there exist a set $X$, a  $\sigma$-ideal $\mathcal{I}$ of subsets of $X$ and an injective $\sigma$-continuous $\ell$-homomorphism $\iota\colon G\hookrightarrow \RXI$.
	
	Let $S$ be a countable generating set of $G$ and let $F\df\{ \lvert s_0\rvert + \cdots + \lvert s_{n-1}\rvert\mid s_0,\dots,s_{n-1}\in S\}$.
	Let us enumerate $F$ as $F=\{f_0,f_1,f_2,\dots \}$.
	We shall prove that the set $\left\{\iota\left(\overline{f_n}\right)\mid n\in \omega \right\}$, admits a supremum $u\in \RXI$ that satisfies the statement of the lemma.
	
	By Lemma~\ref{l:doppio bigvee}, for each $n\in \omega$, we have $\overline{f_n}=\truncsup_{i\in\omega}^{\overline{f_n}}\left(i\overline{f_n}-\truncsup_{k\in \omega}^{i\overline{f_n}}\overline{f_k}\right)$.	Since $\iota$ is a $\sigma$-continuous $\ell$-homomorphism, using Proposition~\ref{p:sigma-cont}, we have the following.
	\begin{enumerate}
		\item\label{i:1} For each $n\in \omega$, $\iota(f_n)=\truncsup_{i\in\omega}^{\iota(f_n)}\left(i\iota(f_n)-\truncsup_{k\in \omega}^{i\iota(f_n)}\iota\left(\overline{f_k}\right)\right).$
	\end{enumerate}
	For every $n\in \omega$, let $g_n\in \R^X$ be such that $\left[{g}_n\right]_{\mathcal{I}}=\iota\left(\overline{f_n}\right)$.
	Then, by \eqref{i:1}, for $\mathcal{I}$-almost every $x\in X$, the following conditions hold.
	\begin{enumerate}[label = (\arabic*$'$), ref = \arabic*$'$]
		\item \label{i:1prime} For each $n\in\omega$, $g_n(x)=\truncsup_{i\in\omega}^{g_n(x)}\left(ig_n(x)-\truncsup_{k\in \omega}^{ig_n(x)}{g}_k(x)\right).$
	\end{enumerate}
	Let $x$ be such that \eqref{i:1prime} hold.
	Suppose by way of contradiction that $\sup_{n\in \omega  }{g}_n(x)=\infty$.
	Then there exists $n\in \omega$ such that ${g}_n(x)>0$.
	Therefore, we have $g_n(x)=\truncsup_{i\in\omega}^{g_n(x)}\left(ig_n(x)-\truncsup_{k\in \omega}^{ig_n(x)}{g}_k(x)\right)>0$, which implies that there exists $i\in \omega$ such that $ig_n(x)-\truncsup_{k\in \omega}^{ig_n(x)}{g}_k(x)>0$.
	Thus, $\truncsup_{k\in \omega}^{ig_n(x)}{g}_k(x)<ig_n(x)$.
	But $\sup_{n\in \omega  }{g}_n(x)=\infty$ implies $\truncsup_{k\in \omega}^{ig_n(x)}{g}_k(x)=ig_n(x)$, a contradiction.
	Therefore, $\sup_{n\in \omega  }{g}_n(x)\in \R$ holds for each $x\in X$ satisfying (1$'$), and thus for $\mathcal{I}$-almost every $x\in X$.
	By Lemma~\ref{l:admits sup}, the set $\left\{[{g}_n]_\mathcal{I}\mid n\in \omega\right\}=\left\{\iota\left(\overline{f_n}\right)\mid n\in \omega \right\}$ admits a supremum $u$.
	
	Let $f\in G^+$.
	Then,
	\[
		\iota(f)\land u=\iota(f)\land \sup_{n\in \omega  }\iota\left(\overline{f_n}\right)\stackrel{\text{Lem.~\ref{l:distriutivity a la BKW}}}{=}\sup_{n\in \omega  }\left\{\iota(f)\land \iota\left(\overline{f_n}\right) \right\}=\sup_{n\in \omega  }\left\{\iota\left(f\land \overline{f_n}\right) \right\}\stackrel{\text{\eqref{i:T3}}}{\leq}\iota\left(\overline{f}\right).
	\]
	For the opposite inequality, by Lemma~\ref{l:bounded}  there exists $m\in \omega$ such that $\overline{f}\leq f_m$.
	Then $\overline{f}=\overline{f}\land f_m\stackrel{\text{\eqref{i:T3}}}{\leq}\overline{f_m}$.
	Therefore $\iota\left(\overline{f}\right)\leq \iota\left(\overline{f_m}\right)\leq u$, and moreover $\iota\left(\overline{f}\right)\leq \iota(f)$ by \eqref{i:T3}.
	Thus, $\iota\left(\overline{f}\right)\leq \iota(f)\land u$.
	For an arbitrary $f\in G$, $\overline{f}=\overline{f^+}-f^-$ by \eqref{i:T1}, hence $\iota\left(\overline{f}\right)=\iota\left(\overline{f^+}\right)-\iota\left(f^-\right)=\left(\iota\left(f^+\right)\land u\right)-\iota\left(f^-\right)\stackrel{\text{Lem.\ }\ref{l:T1}}{=}\iota(f)\land u$.
\end{proof}
Let $G$ be a Dedekind $\sigma$-complete $\ell$-group, let $H\seq G$, and let $u\in G$.
We say that $u$ is a \emph{weak unit for $H$} if $u\geq 0$ and, for every $h\in H$,
\[
	\lvert h\rvert=\truncsup_{n\in \omega}^{\lvert h\rvert}n(\lvert h\rvert\land u).
\]
\begin{remark}
	We will see in Lemma~\ref{l:equation for weak unit} that a weak unit for $G$ in the sense above is the same as a weak unit of $G$ in the usual sense.
\end{remark}
\begin{lemma}\label{l:may be rescaled to 1}
	Let $Y$ be a set, $\mathcal{J}$ a $\sigma$-ideal of subsets of $Y$, $H\seq \frac{\R^Y}{\mathcal{J}}$ an $\ell$-subgroup, and $u\in \frac{\R^Y}{\mathcal{J}}$ a weak unit for $H$.
	Then, there exists a set $X$, a $\sigma$-ideal $\mathcal{I}$ of subsets of $X$, and a $\sigma$-continuous $\ell$-homomorphism $\psi\colon \frac{\R^Y}{\mathcal{J}}\to \RXI$ such that the restriction of $\psi$ to $H$ is injective and $\psi(u)=[1]_\mathcal{I}$.
\end{lemma}
\begin{proof}
	Let $v\in \R^Y$ be such that $[v]_{\mathcal{J}}=u$.
	Since $u\geq 0$, we may choose $v\geq 0$.
	Let $X\df\{y\in Y\mid v(y)>0 \}$.
	Let $\mathcal{I}\df \{J\cap X\mid J\in \mathcal{J} \}=\{J\in\mathcal{J}\mid J\seq X \}$.
	Let $(\;\cdot\;)_{\mid X}\colon \R^Y\to \R^X$ be the restriction map that sends $f\in \R^Y$ to $f_{\mid X}\in \R^X$, where $f_{\mid X}(x)=f(x)$ for each $x\in X$.
	Write $[\;\cdot\;]_\mathcal{J}\colon \R^Y\twoheadrightarrow \frac{\R^Y}{\mathcal{J}}$ for the natural quotient map, and similarly for $[\;\cdot\;]_\mathcal{I}\colon \R^X\twoheadrightarrow \frac{\R^X}{\mathcal{I}}$.
	Since $\ker{[\;\cdot\;]_\mathcal{J}}\subseteq \ker{\left([\;\cdot\;]_\mathcal{I}\circ (\;\cdot\;)_{\mid X}\right)}$, by the universal property of the quotient there exists a unique $\sigma$-continuous $\ell$-homomorphism $\rho\colon \frac{\R^Y}{\mathcal{J}}\to \RXI$ such that the following diagram commutes.
	\[ 
		\begin{tikzpicture}[node distance=1.7 cm, auto]
			\node (11) {$\R^Y$};
			\node (12) [right of=11]{$\R^X$};
			\node (21) [below of=11]{$\frac{\R^Y}{\mathcal{J}}$};
			\node (22) [right of=21]{$\RXI$};
			
			{\draw[->>] (11) to node {$(\;\cdot\;)_{\mid X}$}  (12);}
			{\draw[->>] (11) to node [swap]{$[\;\cdot\;]_\mathcal{J}$}  (21);}
			{\draw[->>] (12) to node {$[\;\cdot\;]_\mathcal{I}$}  (22);}
			{\draw[->, dashed] (21) to node {$\rho$}  (22);}
		\end{tikzpicture}
	\]
	We claim that the restriction of $\rho$ to $H$ is injective.
	Indeed, let $h\in H^+$ be  such that $\rho(h)=0$.
	Let $g\in \R^Y$ be such that $[g]_{\mathcal{J}}=h$.
	Since $h\geq 0$, we may choose $g\geq 0$.
	We have that $[g_{\mid X}]_{\mathcal{I}}=0$.
	Therefore, for $\mathcal{I}$-almost every $x\in X$, $g(x)=0$.
	Therefore, for $\mathcal{J}$-almost every $y\in Y$, $g(y)=0$ or $y\in Y\setminus X$, i.e., $g(y)=0$ or $v(y)=0$.
	Since
	$ h=\truncsup_{n\in \omega}^{ h}n( h\land u)$, we have  
	$g(y)=\truncsup_{n\in \omega}^{ g(y)}n( g(y)\land v(y))$ for $\mathcal{J}$-almost every $y\in Y$.
	Therefore, for $\mathcal{J}$-almost every $y\in Y$, if $v(y)=0$, then $ g(y)=\truncsup_{n\in \omega}^{ g(y)}n( g(y)\land 0)=\truncsup_{n\in \omega}^{ g(y)}0=0$, i.e.\ $g(y)=0$.
	Hence, for $\mathcal{J}$-almost every $y\in Y$, $g(y)=0$.
	Thus, $h=0$.
	
	For every $\lambda\in \R^+\setminus\{0\}$, the function $\lambda(\;\cdot \;)\colon \R\to \R$ which maps $x$ to $\lambda x$ is an isomorphism of Dedekind $\sigma$-complete $\ell$-groups.
	Indeed, its inverse is the map $\frac
	{1}{\lambda}(\;\cdot \;)$.
	Then, the map $m\colon \R^X\to\R^X$ which maps $f$ to the function $m(f)$ defined by $(m(f))(x)=\frac{1}{v(x)}f(x)$ is an isomorphism of Dedekind $\sigma$-complete $\ell$-groups; indeed, its inverse is $m^{-1}\colon \R^X\to \R^X$ defined by $(m^{-1}(g))(x)=v(x)g(x)$.
	For every $f,g\in \R^X$, $[f]_\mathcal{I}=[g]_\mathcal{I}$ if, and only if, $[m(f)]_\mathcal{I}=[m(g)]_\mathcal{I}$.
	Hence, $\ker{[\;\cdot\;]_\mathcal{J}}=\ker{([\;\cdot\;]_\mathcal{J}\circ m)}$.
	Therefore, there exists an isomorphism $\eta\colon\RXI\xrightarrow{\sim}\RXI$ of Dedekind $\sigma$-complete $\ell$-groups which makes the following diagram commute.
	\[ 
		\begin{tikzpicture}[node distance=1.7 cm, auto]
			\node (11) {$\R^X$};
			\node (12) [right of=11]{$\R^X$};
			\node (21) [below of=11]{$\RXI$};
			\node (22) [right of=21]{$\RXI$};

			{\draw[->>] (11) to node [swap]{$[\;\cdot\;]_\mathcal{I}$}  (21);}
			{\draw[->>] (12) to node {$[\;\cdot\;]_\mathcal{I}$}  (22);}
			{\draw[->] (11) to node {$m$}  (12);}
			{\draw (11) to node [swap] {$\sim$}  (12);}
			{\draw[->, dashed] (21) to node {$\eta$}  (22);}
			{\draw (21) to node [swap]{$\sim$}  (22);}
		\end{tikzpicture}
	\]
	We have the following commutative diagram.
	\[ 
		\begin{tikzpicture}[node distance=1.7 cm, auto]
			\node (11) {$\R^Y$};
			\node (12) [right of=11]{$\R^X$};
			\node (21) [below of=11]{$\frac{\R^Y}{\mathcal{J}}$};
			\node (22) [right of=21]{$\RXI$};
			
			\node (13) [right of=12]{$\R^X$};
			\node (23) [right of=22]{$\RXI$};
			
			{\draw[->>] (11) to node {$(\;\cdot\;)_{\mid X}$}  (12);}
			{\draw[->>] (11) to node [swap]{$[\;\cdot\;]_\mathcal{J}$}  (21);}
			{\draw[->>] (12) to node {$[\;\cdot\;]_\mathcal{I}$}  (22);}
			{\draw[->] (21) to node {$\rho$}  (22);}
			
			{\draw[->>] (13) to node {$[\;\cdot\;]_\mathcal{I}$}  (23);}
			{\draw[->] (12) to node {$m$}  (13);}
			{\draw (12) to node [swap] {$\sim$}  (13);}
			{\draw[->] (22) to node {$\eta$}  (23);}
			{\draw (22) to node [swap]{$\sim$}  (23);}
		\end{tikzpicture}
	\]
	We set $\psi\df\eta\circ\rho$.
	Note that $m(v_{\mid X})\in \R^X$ is the function constantly equal to $1$: indeed, $m(v_{\mid X})(x)=\frac{1}{v(x)}v_{X}(x)=1$.
	Thus, $\psi(u)=\eta(\rho(u))=\eta(\rho([v]_{\mathcal{J}}))=[m(v_{X})]_{\mathcal{I}}=[1]_\mathcal{I}$.
	Since the restriction of $\rho$ to $H$ is injective, and $\eta$ is bijective, the restriction of $\psi$ to $H$ is injective.
\end{proof}
\begin{lemma}\label{l:L-S countably}
	Let $G$ be a countably generated Dedekind $\sigma$-complete truncated $\ell$-group.
	Then there exist a set $X$, a  $\sigma$-ideal $\mathcal{I}$ of subsets of $X$ and an injective $\sigma$-continuous $\ell$-homomorphism $\iota\colon G\hookrightarrow \RXI$ such that, for every $f\in G$, $\iota(\overline{f})=\iota(f)\land [1]_\mathcal{I}$.
\end{lemma}
\begin{proof}
	By Lemma~\ref{l:the big step}, there exist a set $Y$, a  $\sigma$-ideal $\mathcal{J}$ of subsets of $Y$, an injective $\sigma$-continuous $\ell$-homomorphism $\varphi\colon G\hookrightarrow \frac{\R^Y}{\mathcal{J}}$ and an element $u\in \frac{\R^Y}{\mathcal{J}}$ such that, for every $f\in G$, 
	\[
	\varphi\left(\overline{f}\right)=\varphi(f)\land u.
	\]
	First, $0\leq \varphi(\overline{0})=0\land u$, hence $u\geq0$.
	Since, for all $f\in G$, $\lvert f\rvert=\truncsup_{n\in \omega}^{\lvert f\rvert}n{\overline{\lvert f\rvert}}$ by \eqref{i:T4prime}, we have $\lvert \varphi(f)\rvert=\truncsup_{n\in \omega}^{\lvert \varphi(f)\rvert}n(\lvert \varphi(f)\rvert \land u)$.
	Therefore, setting $H$ equal to the image of $G$, $u$ is a weak unit for $H$.
	By Lemma~\ref{l:may be rescaled to 1}, there exist a set $X$, a $\sigma$-ideal $\mathcal{I}$ of subsets of $X$, and a $\sigma$-continuous $\ell$-homomorphism $\psi\colon \frac{\R^Y}{\mathcal{J}}\to \RXI$ such that the restriction of $\psi$ to $H$ is injective and $\psi(u)=[1]_\mathcal{I}$.
	The function $\iota\df \psi\circ \varphi$ has the required properties.
\end{proof}
\begin{theorem}\label{t:R generates}
	The variety $\Gt$ of Dedekind $\sigma$-complete truncated $\ell$-groups is generated by $\R$.
\end{theorem}
\begin{proof}
	Let $G$ be a Dedekind $\sigma$-complete truncated $\ell$-group.
	Suppose that an equation $\tau=\rho$ (in the language of Dedekind $\sigma$-complete truncated $\ell$-groups) does not hold in $G$.
	Since $\tau$ and $\rho$ have countably many arguments, the equation $\tau=\rho$ does not hold in a countably generated Dedekind $\sigma$-complete truncated $\ell$-group $G'$.
	By Lemma~\ref{l:L-S countably}, $\tau=\rho$ does not hold in $\R$.
	The statement follows by the HSP Theorem for (infinitary) varieties (see Theorem~(9.1) in \cite{Slominski}).
\end{proof}
\begin{proof}[Proof of Theorem~\ref{t:LS gr}]
	Since the variety of Dedekind $\sigma$-complete truncated $\ell$-groups is generated by $\R$, there exists a set $X$, a $\Gt$-subalgebra $H\seq \R^X$, and a surjective morphism $\psi\colon H\to G$ of Dedekind $\sigma$-complete truncated $\ell$-groups.
	Let 
	\[
		\mathcal{I}\df\{A\seq X\mid \exists (f_n)_{n\in \omega}\seq \ker{\psi}: \forall a\in A\ \exists n\in\omega \text{ s.t.\ }f_n(a)\neq 0 \}.
	\]
	Note that $\mathcal{I}$ is a $\sigma$-ideal of subsets of $X$.
	Therefore we have the projection map $\R^X\to\RXI$ which is a morphism of Dedekind $\sigma$-complete truncated $\ell$-groups.
	If $f\in \ker{\psi}$, then $f(x)=0$ for $\mathcal{I}$-almost every $x\in X$.
	In other words, if $f\in \ker{\psi}$, then $[f]_{\mathcal{I}}=0$.
	For the universal property of quotients, there exists a morphism  $\iota\colon G\to \frac{R^X}{\mathcal{I}}$ of Dedekind $\sigma$-complete truncated $\ell$-groups such that the following diagram commutes.
	\[ 
		\begin{tikzpicture}[node distance=1.7 cm, auto]
			
			\node (00) {H};
			\node (01) [right of=00] {$\R^X$};
			\node (10) [below of=00]{$G$};
			\node (11) [right of=10]{$\RXI$};
			
			{\draw[right hook->] (00) to node {$\iota$}  (01);}
			{\draw[->>] (00) [swap]to node {$\psi$}  (10);}
			{\draw[->>] (01) to node {$[\;\cdot\;]_\mathcal{I}$}  (11);}
			{\draw[->] [dashed](10) to node {$\iota$}  (11);}
			
		\end{tikzpicture}
	\]
	Let $f\in H$ be such that $\iota(\psi(f))=[f]_{\mathcal{I}}=0$.
	Then there exists a set $A\in \mathcal{I}$ such that $f(x)=0$ for every $x\in X\setminus A$.
	Since $A\in \mathcal{I}$, there exists a sequence $(f_n)_{n\in \omega}$ of elements of $\ker{\psi}$ such that, for every $a \in A$, there exists $n \in \omega$ such that $f_n(a) \neq 0$.
	Let us show
	\begin{equation}\label{e:k,n}
		\lvert f\rvert =\truncsup_{n,k\in \omega}^{\lvert f\rvert}k\lvert f_n\rvert.
	\end{equation}
	Equation \eqref{e:k,n} holds if, and only if, for every $a\in X$, $\lvert f(a)\rvert =\truncsup_{n,k\in \omega}^{\lvert f(a)\rvert}k\lvert f_n(a)\rvert$.
	If $a\notin A$, then both sides equal $0$.
	If $a\in A$, then there exists $m\in \omega$ such that $f_m(a)\neq 0$, and therefore $\truncsup_{n,k\in \omega}^{\lvert f(a)\rvert}k\lvert f_n(a)\rvert\geq \truncsup_{k\in \omega}^{\lvert f(a)\rvert}k\lvert f_m(a)\rvert=\lvert f(a)\rvert$.
	Since the opposite inequality is trivial, \eqref{e:k,n} is shown.
	By \eqref{e:k,n}, 
	\[
		\lvert \psi(f)\rvert =\truncsup_{n,k\in \omega}^{\lvert \psi(f)\rvert}k\lvert \psi(f_n)\rvert\stackrel{f_n\in \ker{\psi}}{=}\truncsup_{n,k\in \omega}^{\lvert \psi(f)\rvert}0=0.
	\]
	Therefore $\psi(f)=0$, and thus $f\in \ker{\psi}$.
	This implies that $\iota$ is injective.
\end{proof}
%

%%%%%%%%%%%%%%%%%%%%%%%%%%%%%%%%%%%%%%%%%%%%%%%%%%%%%%%%%%%%%%%%%%%%%%%%%%%%%%%%%%%%
%SECTION
\section{$\R$ generates Dedekind $\sigma$-complete truncated Riesz spaces}\label{s:truncated are variety}

\begin{theorem}[Loomis-Sikorski for truncated Riesz spaces]\label{t:LS RS}
	Let $G$ be a  Dedekind $\sigma$-complete truncated Riesz space.
	Then there exist a set $X$, a  $\sigma$-ideal $\mathcal{I}$ of subsets of $X$, and an injective $\sigma$-continuous Riesz morphism $\iota\colon G\hookrightarrow \RXI$ such that, for every $f\in G$, $\iota\left(\overline{f}\right)=\iota(f)\land [1]_\mathcal{I}$.
\end{theorem}
\begin{proof}
	By Theorem~\ref{t:LS gr}, there exist a set $X$, a  $\sigma$-ideal $\mathcal{I}$ of subsets of $X$, and an injective $\sigma$-continuous $\ell$-homomorphism $\iota\colon G\hookrightarrow \RXI$ such that, for every $f\in G$, $\iota\left(\overline{f}\right)=\iota(f)\land [1]_\mathcal{I}$.
	Since $\RXI$ is Dedekind $\sigma$-complete, it is archimedean; by Corollary~11.53 in \cite{Handana}, $\iota$ is a Riesz morphism.
\end{proof}
We denote by $\RSt$ the variety of Dedekind $\sigma$-complete truncated Riesz spaces, whose primitive operations are $0$, $+$, $\lor$, $\lambda ({\;\cdot\;})$ (for each $\lambda \in \R$),  $\truncsup$, and $\trunc$, and whose axioms are the axioms of Riesz spaces, together with \eqref{i:TS1}, \eqref{i:TS2}, \eqref{i:TS3}, \eqref{i:T1}, \eqref{i:T2}, \eqref{i:T3}, \eqref{i:T4prime} and \eqref{i:T5prime}.

We can now obtain the first main result of Part~\ref{part:variety}, as a consequence of Theorem~\ref{t:LS RS}.
\begin{theorem}\label{t:variety is truncated}
	The variety $\RSt$ of Dedekind $\sigma$-complete truncated Riesz spaces is generated by $\R$.
\end{theorem}
\begin{proof}
	Let $G$ be a Dedekind $\sigma$-complete truncated Riesz space.
	By Theorem~\ref{t:LS RS}, there exist a set $X$, a  $\sigma$-ideal $\mathcal{I}$ of subsets of $X$, and an injective $\sigma$-continuous Riesz morphism $\iota\colon G\hookrightarrow \RXI$ such that, for every $f\in G$, $\iota\left(\overline{f}\right)=\iota(f)\land [1]_\mathcal{I}$.
	Regarding $\RXI$ as an object of $\RSt$ with the structure induced from $\R$, we conclude that $G$ is a subalgebra of a quotient of a power of $\R$.
\end{proof}
\begin{remark}\label{r:quasi-variety}
	From Theorem~7.4 in \cite{Abba}, it follows that $\R$ actually generates $\RSt$ as a quasi-variety, where quasi-equations are allowed to have countably many premises only.
\end{remark}
\begin{corollary}\label{c:main theorem 2}
	For any set $I$,
	\begin{equation*}
		\begin{split}
			\mathrm{F_t}(I) \df	\{	& f \colon \R^I \to \R \mid f \text{ is } \Cyl\left(\R^I\right) \text{-measurable and } \\
									& \exists J \seq I \text{ finite}, \exists(\lambda_j)_{j \in J} \seq \R^+: \lvert f \rvert \leq \sum_{j \in J} \lambda_{j} \lvert \pi_j\rvert \} = \\
								=\{	& f \colon \R^I \to \R \mid f \text{ preserves integrability}\}
		\end{split}
	\end{equation*}
	is the Dedekind $\sigma$-complete truncated Riesz space freely generated by the projections  $\pi_i\colon \R^I\to \R$ ($i\in I$).
\end{corollary}
\begin{proof}
	By Theorem~\ref{t:variety is truncated}, the variety $\RSt$ of Dedekind $\sigma$-complete truncated Riesz spaces is generated by $\R$.
	Therefore, by a standard result in general algebra, the smallest $\RSt$-subalgebra $S$ of $\R^{\R^I}$ that contains the set of projection functions $\{\pi_i\colon \R^I\to \R\mid i\in I \}$ is freely generated by the projection functions.
	The set $S$ is the smallest subset of $\R^{\R^I}$ that contains, for each $i\in I$, the projection function $\pi_i\colon \R^I\to \R$, and which is closed under every primitive operation of $\RSt$.
	 By Theorem~\ref{t:generation finite}, $S$ consists precisely of all operations $\R^I\to \R$ that preserve integrability.
	An application of Theorem~\ref{t:MAIN} completes the proof.
\end{proof}
Write $\pi\colon I\to \mathrm{F_t}(I)$ for the function $\pi(i)=\pi_i$.
Then, Corollary~\ref{c:main theorem 2} asserts the following.
For any set $I$, for every Dedekind $\sigma$-complete truncated Riesz space $G$, for every function $f\colon I\to G$, there exists a unique $\sigma$-continuous truncation-preserving Riesz morphism $\varphi\colon \mathrm{F_t}(I)\to G$ such that the following diagram commutes.
\[
	\begin{tikzpicture}[node distance=1.7 cm, auto]
		\node (11) {$I$};
		\node (12) [right of=11]{$\mathrm{F_t}(I)$};
		\node (22) [below of=12]{$G$};
		
		{\draw[->] (11) to node {$\pi$}  (12);}
		{\draw[->] (11) to node [swap]{$\forall f$}  (22);}
		{\draw[->, dashed] (12) to node {$\exists!	\varphi$}  (22);}
	\end{tikzpicture}
\]
%

%%%%%%%%%%%%%%%%%%%%%%%%%%%%%%%%%%%%%%%%%%%%%%%%%%%%%%%%%%%%%%%%%%%%%%%%%%%%%%%%%%%%
%SECTION
\section{The Loomis-Sikorski Theorem for $\ell$-groups with weak unit}\label{s:LS weak}

An element $1$ of an abelian $\ell$-group $G$  is a \emph{weak unit} if $1	\geq 0$ and, for every $f\in G$, $f\land 1=0$ implies $f=0$.
\begin{remark}\label{r:weak is trunc}
	Let $G$ be an archimedean abelian $\ell$-group, and let $1$ be a weak unit.
	Then $f\mapsto f\land 1$ is  a truncation.
	Indeed, the following show that (\ref{i:T1}-\ref{i:T5}) hold.
	\begin{enumerate}
		\item $f\land u=(f^+\land u)-f^-$ by Lemma~\ref{l:T1}.
		\item For all $f\in G^+$, $f\land 1\in G^+$.
		\item For all $f,g\in G^+$, $f\land (g\land 1)=(f\land 1)\land g\leq f\land 1\leq f$.
		\item For all $f\in G^+$, if $f\land 1=0$, then $f=0$.
		\item For all $f\in G^+$, if $nf=(nf)\land 1$ for every $n\in \omega$, then $nf\leq 1$ for every $n\in \omega$.
	Then, since $G$ is archimedean, $f=0$.
	\end{enumerate}
\end{remark}
\begin{lemma}\label{l:equation for weak unit}
	Let $G$ be a  Dedekind $\sigma$-complete $\ell$-group $G$, and let $1\in G$.
	Then, $1$ is a weak unit if, and only if, the following conditions hold.
	\begin{enumerate}[label = {\rm(W\arabic*)}, ref = W\arabic*]
		\item \label{i:W1} $1\geq 0$.
		\item \label{i:W2} For all $f\in G^+$,
		\[
			f=\truncsup_{n\in \omega}^fn(f\land 1).
		\]
	\end{enumerate}
\end{lemma} 
\begin{proof}
	Since $G$ is Dedekind $\sigma$-complete, $G$ is archimedean.
	If $1$ is a weak unit, then $1\geq 0$ and, by Remark~\ref{r:weak is trunc} and Proposition~\ref{p:B4}, for all $f\in G^+$, $f=\truncsup_{n\in \omega}^fn(f\land 1)$.
	Conversely, suppose that \eqref{i:W1} and \eqref{i:W2} hold.
	If $f\land 1=0$, then $f=\truncsup_{n\in \omega}^fn(f\land 1)=\truncsup_{n\in \omega}^f 0=0$, and so $1$ is a weak unit.
\end{proof}
Note that, in the language of Dedekind $\sigma$-complete $\ell$-groups, axioms \eqref{i:W1} and \eqref{i:W2} are equational.
Indeed, \eqref{i:W1} corresponds to $1 \land 0 = 0$, and \eqref{i:W2} corresponds to $\forall f \ f^+ = \truncsup_{n\in \omega}^{f^+} n\left(f^+\land 1\right)$.
\begin{theorem}[Loomis-Sikorski Theorem for $\ell$-groups with weak unit]\label{t:LS gr weak}
	Let $G$ be a  Dedekind $\sigma$-complete $\ell$-group with weak unit $1$.
	Then there exist a set $X$, a  $\sigma$-ideal $\mathcal{I}$ of subsets of $X$, and an injective $\sigma$-continuous $\ell$-homomorphism $\iota\colon G\hookrightarrow \RXI$ such that $\iota(1)=[1]_\mathcal{I}$.
\end{theorem}
\begin{proof}
	By Remark~\ref{r:weak is trunc}, $G$ is a Dedekind $\sigma$-complete truncated $\ell$-group, with the truncation given by $f\mapsto f\land 1$.
	Then, by Theorem~\ref{t:LS gr}, there exist a set $Y$, a  $\sigma$-ideal $\mathcal{J}$ of subsets of $Y$ and an injective $\sigma$-continuous $\ell$-homomorphism $\varphi\colon G\hookrightarrow \RXI$ such that, for every $f\in G$, $\varphi(f\land 1)=\varphi(f)\land [1]_\mathcal{J}$.
	The element $\varphi(1)$ is a weak unit for the image of $G$ under $\varphi$.
	Therefore, by Lemma~\ref{l:may be rescaled to 1}, there exists a set $X$, a $\sigma$-ideal $\mathcal{I}$ of subsets of $X$, and a $\sigma$-continuous $\ell$-homomorphism $\psi\colon \frac{\R^Y}{\mathcal{J}}\to \RXI$ such that the restriction of $\psi$ to $H$ is injective and $\psi(\varphi(1))=[1]_\mathcal{I}$.
	The function $\iota\df \psi\circ \varphi$ has the desired properties.
\end{proof}
\begin{corollary}
	The variety of Dedekind $\sigma$-complete $\ell$-groups with weak unit is generated by $\R$.
\end{corollary}
\begin{proof}
	Let $G$ be a Dedekind $\sigma$-complete $\ell$-group with weak unit.
	By Theorem~\ref{t:LS gr weak}, $G$ is a subalgebra of a quotient of a power of $\R$.
\end{proof}
%

%%%%%%%%%%%%%%%%%%%%%%%%%%%%%%%%%%%%%%%%%%%%%%%%%%%%%%%%%%%%%%%%%%%%%%%%%%%%%%%%%%%%
%SECTION
\section{$\R$ generates Dedekind $\sigma$-complete Riesz spaces with weak unit}\label{s:weak generated R}

\begin{theorem}[Loomis-Sikorski for Riesz spaces with weak unit]\label{t:LS RS weak}
	Let $G$ be a  Dedekind $\sigma$-complete Riesz space with weak unit.
	Then there exist a set $X$, a  $\sigma$-ideal $\mathcal{I}$ of subsets of $X$, and an injective $\sigma$-continuous Riesz morphism $\iota\colon G\hookrightarrow \RXI$ such that $\iota(1)=[1]_{\mathcal{I}}$.
\end{theorem}
\begin{proof}
	By Theorem~\ref{t:LS RS}, there exist a set $X$, a  $\sigma$-ideal $\mathcal{I}$ of subsets of $X$ and an injective $\sigma$-continuous $\ell$-homomorphism $\iota\colon G\hookrightarrow \RXI$ such that, for every $f\in G$, $\iota(1)= [1]_{\mathcal{I}}$.
	Since $\RXI$ is Dedekind $\sigma$-complete, and thus archimedean, by Corollary~11.53 in \cite{Handana}, $\iota$ is a Riesz morphism.
\end{proof}
We denote by $\RSu$ the variety of Dedekind $\sigma$-complete Riesz spaces with weak unit, whose primitive operations are $0$, $+$, $\lor$, $\lambda ({\;\cdot\;})$ (for each $\lambda \in \R$),  $\truncsup$, and $1$, and whose axioms are the axioms of Riesz spaces, together with \eqref{i:TS1}, \eqref{i:TS2}, \eqref{i:TS3}, \eqref{i:W1}, \eqref{i:W2}.

As the second main result of Part~\ref{part:variety}, we now deduce a theorem that was already obtained in \cite{Abba}.
\begin{theorem}\label{t:variety is truncated weak}
	The variety $\RSu$ of Dedekind $\sigma$-complete Riesz spaces with weak unit is generated by $\R$.
\end{theorem}
\begin{proof}
	Let $G$ be a Dedekind $\sigma$-complete truncated Riesz space.
	By Theorem~\ref{t:LS RS weak}, $G$ is a subalgebra of a quotient of a power of $\R$.
\end{proof}
\begin{remark}
	It was shown in \cite{Abba} that $\R$ actually generates $\RSu$ as a quasi-variety, in the sense of Remark~\ref{r:quasi-variety}.
\end{remark}
\begin{corollary}\label{c:main theorem 1}
	For any set $I$,
	\begin{equation*}
		\begin{split}
			\mathrm{F_u}(I)\df	\{	& f \colon \R^I \to \R \mid f \text{ is }\Cyl\left(\R^I\right) \text{-measurable and }\\
									& \exists J \seq I \text{ finite}, \exists(\lambda_j)_{j \in J} \seq \R^+, \exists k \in \R^+: \lvert f \rvert \leq k + \sum_{j \in J} \lambda_{j} \lvert \pi_j\rvert \} = \\
								=\{	& f \colon \R^I \to \R \mid f \text{ preserves integrability over finite measure spaces}\}
		\end{split}
	\end{equation*}
	is the Dedekind $\sigma$-complete Riesz space with weak unit freely generated by the elements $\{\pi_i \}_{i\in I}$, where, for $i\in I$,  $\pi_i\colon \R^I\to \R$ is the projection on the $i$-th coordinate.
\end{corollary}
The proof is analogous to the proof of Corollary~\ref{c:main theorem 2}, and $\mathrm{F_u}(I)$ is characterised by a universal property analogous to the one that characterises $\mathrm{F_t}(I)$.

\appendix
%%%%%%%%%%%%%%%%%%%%%%%%%%%%%%%%%%%%%%%%%%%%%%%%%%%%%%%%%%%%%%%%%%%%%%%%%%%%%%%%%%%%
%SECTION
\section{Operations that preserve $\infty$-integrability}\label{s:infty}

In Section~\ref{s:lp} it has been shown that, for any $p\in [1,+\infty)$, a function $\tau\colon \R^I\to \R $  preserves $p$-integrability if, and only if, $\tau$ is $\Cyl\left(\R^I\right)$-measurable and there exist a finite subset of indices $J \seq I$ and nonnegative real numbers $(\lambda_j)_{j \in J}$ such that, for every $v\in \R^I$, we have $\lvert\tau(v)\rvert\leq \sum_{j \in J} \lambda_{j} \lvert v_j\rvert$.
Does the same hold for $p=\infty$? The answer is no.
Indeed, the function $(\;\cdot\;)^2\colon \R\to \R, x\mapsto x^2$ is an example of operation which preserves $\infty$-integrability but not $p$-integrability, for every $p\in [1,+\infty)$.
In Theorem~\ref{t:pres infty}, we will answer the following question.
\begin{question}\label{q:main infty}
	Which operations $\R^I\to \R$ preserve $\infty$-integrability?
\end{question} 
We will see that an operation $\R^I\to \R$ preserve $\infty$-integrability if, and only if, roughly speaking, it is measurable and it maps coordinatewise-bounded subsets of $\R^I$ onto bounded subsets of $\R$.
To make this precise, we introduce some definitions.

Given a measure space $(\Omega,\mathcal{F},\mu)$, we define $\mathcal{L}^\infty(\mu)$ as the set of $\mathcal{F}$-measurable functions from $\Omega$ to $\R$ that are bounded outside of a measurable set of null $\mu$-measure.
\begin{definition}
	Let $I$ be a set, $\tau\colon \R^I\to \R$.
	We say that $\tau$ \emph{preserves $\infty$-integrability} if for every measure space $(\Omega, \mathcal{F},\mu)$ and every family $(f_i)_{i\in I}\seq \mathcal{L}^\infty(\mu)$ we have  $\tau((f_i)_{i\in I})\in \mathcal{L}^\infty(\mu)$.\qed
\end{definition}
We can now state the answer to Question~\ref{q:main infty} precisely.
Let $I$ be a set and let $\tau\colon \R^I\to \R$ be a function.
Then $\tau$ preserves $\infty$-integrability if, and only if, $\tau$ is $\Cyl\left(\R^I\right)$-measurable and, for every $(M_i)_{i\in I}\seq \R^+$, the restriction of $\tau$ to $\prod_{i\in I}[-M_i,M_i]$ is bounded.
This will follow from Theorem~\ref{t:pres infty}.

%===================================================================================
%SUBSECTION
\subsection{Operations that preserve boundedness}

As a preliminary step, in Theorem~\ref{t:boundedness}, we characterise the operations which preserve boundedness.
\begin{definition}
	Let $I$ be a set, $\tau\colon \R^I\to \R$.
	We say that $\tau$ \emph{preserves boundedness} if for every set $\Omega$ and every family $(f_i)_{i\in I}$ of bounded functions $f_i\colon \Omega\to \R$, we have that $\tau((f_i)_{i\in I})\colon \Omega\to \R$ is also bounded.\qed
\end{definition}
\begin{theorem}\label{t:boundedness}
	Let $I$ be a set and $\tau\colon \R^I\to \R$.
	The following conditions are equivalent.
	\begin{enumerate}	
		\item  $\tau$ preserves boundedness.
		\item For every $(M_i)_{i\in I}\seq \R^+$, the restriction of $\tau$ to $\prod_{i\in I}[-M_i,M_i]$ is bounded.
	\end{enumerate}
\end{theorem}
\begin{proof}
	We prove [$(1)\Rightarrow (2)$].
	Fix $(M_i)_{i\in I}\seq \R^+$.
	Take $\Omega\df\prod_{i\in I}[-M_i,M_i]$ and, for every $i\in I$, let $f_i$ be the restriction of the projection function $\pi_i\colon \R^I\to \R$ to $\Omega$.
	Since $f_i$ maps $\Omega$ onto $[-M_i,M_i]$, $f_i$ is bounded.
	Thus $\tau((f_i)_{i\in I})$ is bounded, i.e., there exists $\tilde{M}$ such that for every $x\in \Omega$ we have $\tau((f_i(x))_{i\in I})\in [-\tilde{M}, \tilde{M}]$.
	Let $x\in \Omega$.
	Then $\tau(x)=\tau((\pi_i(x))_{i\in I})=\tau((f_i(x))_{i\in I})\in [-\tilde{M}, \tilde{M}]$.
	Thus (2) holds.
	
	We now prove [$(2)\Rightarrow (1)$].
	Let $\Omega$ be a set, and $(f_i)_{i\in I}$ be a family of bounded functions from $\Omega$ to $\R$.
	For each $i\in I$, let $M_i\in \R^+$ be such that the image of $f_i$ is contained in $[-M_i,M_i]$.
	Let $\tilde{M}$ be such that $\tau$ maps $\prod_{i\in I}[-M_i,M_i]$ onto a subset of $[-\tilde{M},\tilde{M}]$.
	Then, for each $x\in \Omega$, $\tau((f_i)_{i\in I})(x)= \tau((f_i(x))_{i\in I})\in [-\tilde{M},\tilde{M}]$.
\end{proof}
%

%===================================================================================
%SUBSECTION
\subsection{Operations that preserve $\infty$-integrability}

The following is the main theorem of this section.
\begin{theorem}\label{t:pres infty}
	Let $I$ be a set and let $\tau\colon \R^I\to \R$ be a function.
	The following conditions are equivalent.
	\begin{enumerate}
		\item $\tau$ preserves $\infty$-integrability.
		\item $\tau$ preserves measurability and boundedness.
		\item $\tau$ is $\Cyl\left(\R^I\right)$-measurable and, for every $(M_i)_{i\in I}\seq \R^+$, the restriction of $\tau$ to $\prod_{i\in I}[-M_i,M_i]$ is bounded.
	\end{enumerate}
\end{theorem}
In order to prove Theorem~\ref{t:pres infty}, we need some lemmas.
\begin{lemma}\label{l:infty implies meas}
	Let $I$ be a set and let $\tau\colon \R^I\to \R$ be a function.
	If $\tau$ preserves $\infty$-integrability, then $\tau$ preserves measurability.
\end{lemma}
\begin{proof}
	Every measurable space $(\Omega, \mathcal{F})$ may be endowed with the null measure $\mu_0$: for each $A\in \mathcal{F}$, $\mu_0(A)=0$.
	Then $\mathcal{L}^\infty(\mu_0)$ is the set of $\mathcal{F}$-measurable functions from $\Omega$ to $\R$.
	Hence, preservation of $\infty$-integrability over  $(\Omega, \mathcal{F}, \mu_0)$ is equivalent to preservation of measurability over $(\Omega, \mathcal{F})$.
\end{proof}
\begin{lemma}\label{l:infty implies bound}
	Let $I$ be a set and let $\tau\colon \R^I\to \R$ be a function.
	If $\tau$ preserves $\infty$-integrability, then $\tau$ preserves boundedness.
\end{lemma}
\begin{proof}
	Let us suppose that $\tau$ does not preserve boundedness.
	By Theorem~\ref{t:boundedness}, there exists $(M_i)_{i\in I}\seq \R^+$ such that the restriction of $\tau$ to $\prod_{i\in I}[-M_i,M_i]$ is not bounded.
	Fix one such family $(M_i)_{i\in I}$ and let $\Omega\df \prod_{i\in I}[-M_i,M_i]$.
	Let $(\omega_n)_{n\in \omega}$ be a sequence in $\Omega$ such that  $\lvert \tau(\omega_0) \rvert < \lvert \tau(\omega_1)\rvert < \cdots$ and $\lvert \tau(\omega_n)\rvert \to \infty$ as $n\to \infty$.
	Consider on $(\Omega, \mathcal{P}(\Omega))$ the discrete measure $\mu$ such that $\mu(\{\omega_n\})=\frac{1}{2^n}$ for every $n\in \omega$ and $\mu(\Omega\setminus \{\omega_0,\omega_1,\dots\})=0$.
	Then $(\Omega,\mathcal{P}(\Omega),\mu)$ is a finite measure space.
	For $i\in I$, the restriction $(\pi_{i})_{\mid \Omega}$ of $\pi_i$ to $\Omega$ is bounded, since its image is $[-M_i,M_i]$.
	Moreover, $(\pi_{i})_{\mid \Omega}$ is $\mathcal{P}(\Omega)$-measurable.
	Therefore, $(\pi_{i})_{\mid \Omega}\in \mathcal{L}^\infty(\mu)$.
	 We have $\tau_{\mid \Omega}\notin  \mathcal{L}^\infty(\mu) $; indeed, let $A$ be a subset of $\Omega$ of null $\mu$-measure.
	Then $\omega_n\notin A$ for every $n\in \omega$.
	Therefore $\tau_{\mid \Omega}$ is not bounded outside of $A$.
\end{proof}
\begin{lemma}\label{l:meas and bound then infty}
	Let $I$ be a set and let $\tau\colon \R^I\to \R$ be a function.
	If $\tau$ preserves measurability and boundedness, then $\tau$ preserves $\infty$-integrability.
\end{lemma}
\begin{proof}
	By Corollary~\ref{c:meas then count}, $\tau$ depends on a countable subset $J\seq I$.
	Let $(\Omega, \mathcal{F},\mu)$ be a finite measure space and consider a family $(f_i)_{i\in I}\seq \mathcal{L}^\infty(\mu)$.
	For each $j\in J$, let $A_j$ be a  measurable subset of $\Omega$ such that $\mu(A_j)=0$ and $f_j$ is bounded outside of $A_j$.
	Set $A\df \bigcup_{j\in J}A_{j}$.
	Then $\mu\left(A\right)=0$.
	For each $i\in I$, define $\tilde{f_i}$ as $f_i$ if $i\in J$, otherwise let $\tilde{f_i}$ be the function constantly equal to $0$.
	Since $\tau$ depends only on $J$, we have $\tau((f_i)_{i\in I})=\tau((\tilde{f_i})_{i\in I})$.For every $i\in I$, the restriction $(\tilde{f_i})_{\mid \Omega\setminus A}$ is bounded.
	We have $\tau((f_i)_{i\in I})_{\mid \Omega\setminus A}=\tau((\tilde{f_i})_{i\in I})_{\mid \Omega\setminus A}=\tau(((\tilde{f_i})_{\mid \Omega \setminus A})_{i\in I})$ is bounded since $\tau$ preserves boundedness and, for every $i\in I$, $(\tilde{f_i})_{\mid \Omega\setminus A}$ is bounded.
	Thus $ \tau((f_i)_{i\in I})$ is bounded outside of a set of null measure.
	Moreover, $ \tau((f_i)_{i\in I})$ is measurable because $\tau$ preserve measurability.
	Therefore $ \tau((f_i)_{i\in I})\in \mathcal{L}^\infty(\mu)$.
\end{proof}
\begin{proof}[Proof of Theorem~\ref{t:pres infty}]
	By Lemmas~\ref{l:infty implies meas} and \ref{l:infty implies bound}, we have $[(1)\Rightarrow(2)]$.
	Lemma~\ref{l:meas and bound then infty}, we have $[(2)\Rightarrow(1)]$.
	By Theorems~\ref{t:pres measurable} and \ref{t:boundedness}, we have $[(2)\Leftrightarrow(3)]$.	
\end{proof}
\begin{corollary}\label{c:p then infty}
	Let $I$ be a set and let $\tau\colon \R^I\to \R$ be a function.
	If $\tau$ preserves $p$-integrability for some $p\in [1,+\infty)$, then $\tau$ preserves $\infty$-integrability.
\end{corollary}
\begin{proof}
	By Theorem~\ref{t:MAIN}, $\tau$ is $\Cyl\left(\R^I\right)$-measurable and there exist a finite subset of indices $J \seq I$ and nonnegative real numbers $(\lambda_j)_{j \in J}$ such that, for every $v\in \R^I$, we have $\lvert\tau(v)\rvert\leq \sum_{j \in J} \lambda_{j} \lvert v_j\rvert$.
	Let $(M_i)_{i\in I}\seq \R^+$.
	Let $v\in \prod_{i\in I}[-M_i,M_i]$.
	Then $\lvert\tau(v)\rvert\leq \sum_{j \in J} \lambda_{j} \lvert v_j\rvert \leq \sum_{j \in J} \lambda_{j} M_j$.
	Thus, the restriction of $\tau$ to $\prod_{i\in I}[-M_i,M_i]$ is bounded.
	Therefore, by Theorem~\ref{t:pres infty}, $\tau$ preserves $\infty$-integrability.
\end{proof}
\begin{remark}
	The converse of Corollary~\ref{c:p then infty}, as mentioned at the beginning of this section, is not true, as shown by the function $(\;\cdot\;)^2\colon \R\to \R,\, x\mapsto x^2$.
\end{remark}
%

%%%%%%%%%%%%%%%%%%%%%%%%%%%%%%%%%%%%%%%%%%%%%%%%%%%%%%%%%%%%%%%%%%%%%%%%%%%%%%%%%%%%
%BIBLIOGRAPHY

\bibliography{Biblio}
\bibliographystyle{alpha}

\end{document}